\documentclass{article}

\usepackage[a4paper,top=3cm,bottom=3cm,left=2.5cm,right=2.5cm]{geometry}
 
\usepackage[T1]{fontenc}
\usepackage{amsmath}
\usepackage{amssymb}
\usepackage{amsthm}
\usepackage{color}
\usepackage{graphicx}
\usepackage{algpseudocode} 
\usepackage{algorithm} 
\usepackage[]{subfig}
\usepackage{accents}
\usepackage[english]{babel}

\usepackage{multicol}
\usepackage{amsfonts}
\usepackage{graphicx}
\usepackage{caption}
\usepackage{bm}
\usepackage{bbm}  
\usepackage{color}
\usepackage{picture}
\usepackage{tikz}

\renewcommand{\d}{\mathrm{d}}
\newcommand{\dt}{\mathrm{dt}}
\newcommand{\vect}[1]{\boldsymbol{#1}}
\newcommand{\F}{\mathbf{F}}
\newcommand{\kron}{\otimes}
\newcommand{\z}{\phantom{0}}
\newcommand{\dx}{\ d \boldsymbol{x}}

\renewcommand{\S}{\mathcal{S}}

\newcommand{\V}{\mathcal{V}}
\newcommand{\vecx}{\boldsymbol{x}} 

\newcommand{\Vhspaceref}{\widehat{\mathcal{V}}_{s,h_s,0}} 
\newcommand{\Vhtimeref}{\widehat{\mathcal{V}}_{t,h_t,0}} 
\newcommand{\Vhspace}{\mathcal{V}_{s,h_s,0}} 
\newcommand{\Vhtime}{\mathcal{V}_{t,h_t,0}} 

 \newcommand{\omegab}{\boldsymbol{\omega}}
  \newcommand{\mub}{\boldsymbol{\mu}}

\newtheorem{tw}{Theorem}
\newtheorem{lemma}{Lemma}

\newtheorem{rmk}{Remark}
\newtheorem{wn}{Corollary}
\newtheorem{prop}{Proposition}
\newtheorem{ass}{Assumption}

\definecolor{dred}{rgb}{0.92,0,0}

\def\B{\color{black}}
\usepackage{soul}

\definecolor{dorange}{rgb}{0.8,0.3,0}

\begin{document}

\pagestyle{myheadings}
\markboth{M. Montardini, M. Negri, G. Sangalli and M. Tani }{}

\title {\textbf{Space-time least-squares isogeometric method and efficient solver for parabolic problems} \thanks{Version of  \today}}

\author{M. Montardini \thanks{Universit\`a di Pavia, Dipartimento di Matematica ``F. Casorati'',
Via A. Ferrata 1, 27100 Pavia, Italy. }   
\and  M. Negri $^{\dag}$
\and
  G. Sangalli$^{\dag}$\thanks{IMATI-CNR ``Enrico Magenes'',  Pavia, Italy. \vskip 1mm \noindent Emails: 
{\tt  monica.montardini01@universitadipavia.it,  \{matteo.negri, giancarlo.sangalli, mattia.tani\}@unipv.it}}
\and  M. Tani $^{\dag}$}

%
\maketitle
 
\begin{abstract}In this paper, we propose a space-time least-squares isogeometric method
to solve  parabolic evolution problems, well suited  for high-degree smooth
splines in the space-time domain.  We  focus on the linear solver and
its computational efficiency: thanks to the proposed formulation and
to the tensor-product construction of space-time
splines, we can design a
preconditioner whose application requires the solution of a Sylvester-like
equation, which is performed efficiently by the fast diagonalization
method.   The preconditioner is robust w.r.t. spline degree and  
mesh size. The computational time required for its application, for a 
serial execution, is 
 almost proportional to the number of  degrees-of-freedom  and
 independent of the polynomial degree. The proposed approach
 is also well-suited  for parallelization.

\vskip 1mm
\noindent
{\bf Keywords:} Isogeometric analysis, parabolic problem, space-time method,  $k$-method,
splines, least-squares, Sylvester equation.
\end{abstract}

\section{Introduction} 
Isogeometric analysis (IGA) is a recent technique for the numerical
solution of partial differential equations (PDE), introduced in the seminal
paper \cite{Hughes2005}. IGA is an evolution of classical finite
element methods (FEM):  the main idea is to use the same functions (splines
or generalizations) that  represent the computational domain in
Computer-Aided Design systems, also in the approximation of the solution. 
We refer to \cite{Cottrell2009} and  \cite{acta-IGA} for a comprehensive
presentation   and  a mathematical survey of IGA, respectively.

 IGA allows to use high-order and high-smoothness   functions.
The $k$-method, based on  splines  of degree $p$ and  $C^{p-1}$
regularity,  delivers higher accuracy per degree-of-freedom, comparing
to $C^0$ or discontinuous $hp$-FEM \cite{bressan2018approximation,cottrell2007studies,Evans_Bazilevs_Babuska_Hughes}. 
However, the $k$-method  also requires ad-hoc algorithms,  otherwise when the polynomial degree $p$ increases, the
computational cost per degree-of-freedom increases dramatically,  both in the
formation of the matrix and in the solution of the linear system  
\cite{Collier2013,Sangalli2017}.

In this paper, we  design and analyze  an  isogeometric
method for parabolic  equations, focusing on  the heat equation as model problem.  
The most common numerical methods  for time-dependent PDE are obtained  by discretizing  separately  in time (e.g, by difference schemes) and in space (e.g., by a
Galerkin method).  
We consider instead the alternative approach
of discretizing the PDE simultaneously in
space and time, that is, the so-called space-time (variational)
approach.  
A first idea of space-time finite element method  has been introduced
in \cite{fried1969finite,oden1969general,oden1969general2} and
developed for the heat conduction problem in
\cite{bruch1974transient}. Further pioneering studies  on space-time
methods  have been \cite{shakib1991new,hughes1989new}, where the
authors consider a Galerkin formulation and   add a least-squares
operator to enhance stability and mitigate spurious oscillations.
 
More recently, the mathematical analysis of Galerkin
space-time methods for parabolic equations has been  developed  in
\cite{Schwab2009} for a wavelet discretization, and  in
\cite{Steinbach2015} for a Galerkin finite element discretization.

In the IGA framework, the idea of using smooth splines in time has been
first proposed in \cite{takizawa2014space}. The  recent paper 
\cite{takizawa2017turbocharger} applies this concept to a  complex
engineering simulation. A stabilized space-time isogeometric method for
the heat equation  has
been proposed in \cite{Langer2016,Langer2017} and  its  time-parallel multigrid solver has
been developed in \cite{Hofer2018}. 

In contrast to the existing space-time IGA works, in this paper we
adopt  an $L^2$  least-squares approximation. The first appearance of a least-squares  space-time formulation   was
in \cite{nguyen1984space}. However, as discussed in
\cite{bell1994space,bell1996space}, the discretized formulation of \cite{nguyen1984space}
departs from the   least-squares minimization
principle.  In   \cite{bell1994space,bell1996space} the authors consider a
least-squares finite element method for unsteady fluid dynamics
problems. For second-order  differential equations,  the $L^2$ minimization of
the equation residual  would   require  $C^1$-continuous functions
in the spatial    variables, \B however
\cite{bell1994space,bell1996space} recast  the second-order
equation  into a set of first-order 
equations, whose least-squares formulation allows $C^0$
functions. 
 Furthermore, \cite{bell1994space,bell1996space} introduce
a time-marching approach to  lower the memory requirement and the
computational time.
Henceforth,  the most relevant contributions on space-time
least-squares methods  have retained these
two features: 1) the  minimization of first-order residuals and 2) the time-marching technique (similar to the use of time-slabs or discontinuous-in-time approximation).
We refer to the book \cite{bochev2009least} for a review of the
literature. 

Our work departs from the setting described above: we consider
 high degree and smoothness splines in time and space with the
following implications: 1) exploiting the $C^1$-continuity of our
approximating function, we directly minimize the  second-order
residual and 2) we need to solve  a global-in-time linear system.
Point 1) represents an advantage while point 2) is addressed by
exploiting the tensor product structure of the spline basis functions:
we do not need to form the global space-time matrix, which is
given as  sum of Kronecker products of matrices, and we set up a
preconditioner that relies on the solution of a Sylvester-like equation.
Indeed, the  least-squares formulation  allows us to use the same  preconditioning technique
introduced in \cite{Sangalli2016} for the Poisson problem, based  the
so-called fast diagonalization (FD) method  (originally proposed in  \cite{Lynch1964} and more recently discussed
in \cite{Deville2002}). For  the space-time least-squares
formulation, the computational cost of the preconditioner
setup is at most $O(N_{dof})$ floating-point operations (FLOPs)  while its  application  is  $O(N_{dof}^{1 +
  1/d})$ FLOPs, where $d$ is the number of space dimensions and  $N_{dof}$ denotes the
total number of degrees-of-freedom  (for
simplicity, here we consider   the same number of  degrees-of-freedom
in time and in each space direction). In
our numerical benchmarks the measured computational time of the
preconditioner, for  serial single-core
execution,  is close to optimality, that is proportional to $N_{dof}$,
with no  dependence on  $p$. Therefore, the preconditioner is robust with
respect to the polynomial degree. 
Moreover, under the assumption that the coefficients of the equation
do not depend on time, our approach requires a significantly small
amount of memory compared to other space-time approaches: denoting  by
$N_s$  the total number of degrees-of-freedom in space (and assuming the number
of degrees-of-freedom in time is not too large, as in typical applications)  the storage
cost is $O(p^d N_s + N_{dof})$. This is exactly  what one would get for low-order time-marching
schemes.
 
Space-time methods facilitate the full 
parallelization of the solver,  see \cite{dorao2007parallel,Gander2015}. The
preconditioner we propose  fits  in the framework, e.g.,  of
\cite{kvarving2011fast}. We do not address this important
issue in our paper, that will be the focus of our further research.

The paper is organized as follows. In Section 2 we introduce B-Spline
basis functions and the isogeometric spaces that we need for the
discrete analysis. The parabolic model problem is presented in Section
3, where we also discuss the  well-posedness of the least-squares
approximation and the a-priori error estimates. Section 4 focuses on
preconditioning strategy and its spectral analysis. We show
numerical results to assess the performance of the proposed
preconditioner and to confirm the a-priori error estimates in Section
6. Finally, in the last section we draw conclusions and highlight future research directions.

\section{Preliminaries}
\subsection{B-splines}
 A knot vector in $[0,1]$ is a sequence of non-decreasing points
 $\Xi:=\left\{ 0=\xi_1 \leq \dots \leq \xi_{m+p+1}=1\right\}$, where
 $m$ and $p$ are positive integers. We use  open knot vectors, that is
 $\xi_1=\dots=\xi_{p+1}=0$ and $\xi_{m}=\dots=\xi_{m+p+1}=1$. Then, according to Cox-De Boor recursion formulas (see \cite{DeBoor2001}), the univariate B-splines are  piecewise polynomials   $\widehat{b}_{i,p}: (0,1)\rightarrow \mathbb{R}$   defined as
 
\indent
for $p=0$:
 \begin{align*}  
\widehat{b}_{i,0}(\eta) = \begin{cases}1 &  { \textrm{if }} \xi_{i}\leq \eta<\xi_{i+1},\\
0 & \textrm{otherwise,}
\end{cases} 
\end{align*}  
\indent
for $p \geq 1$:
\begin{align*}
\widehat{b}_{i,p}(\eta)= \begin{cases}\dfrac{\eta-\xi_{i}}{\xi_{i+p}-\xi_{i}}\widehat{b}_{  i,p-1}(\eta)   +\dfrac{\xi_{i+p+1}-\eta}{\xi_{i+p+1}-\xi_{i+1}}\widehat{b}_{  i+1,p-1}(\eta)   & { \textrm{if }} \xi_{i}\leq  \eta<\xi_{i+p+1}, \\[8pt]
0  & \textrm{otherwise,}
\end{cases}
\end{align*} 

where we adopt the convention $0/0=0$. We define the univariate spline space as
\[
\widehat{\S}_h^p : = \mathrm{span}\{\widehat{b}_{i,p}\}_{i = 1}^m,
\]
where $h$ denotes the mesh size, {i.e. $h:=\max_{i=1,\dots,m+p}\{|\xi_{i+1}-\xi_i| \}$. }
The smoothness of the B-splines at the knots depends   on the knot
multiplicity  (for more details on B-splines and their use in
isogeometric analysis, see \cite{Cottrell2009}  and  \cite{DeBoor2001}). 

Multivariate B-splines are defined as tensor product of univariate
B-splines.    We will consider functions of space and time, where
the space domain is $d$-dimensional.  Even if the analysis works for a  general $d$, in the numerical tests  we will focus on $d=2,3$, which are the most interesting cases in practical applications. 
Therefore we introduce $d+1 $
univariate knot vectors $\Xi_l:=\left\{ \xi_{l,1} \leq \dots \leq
  \xi_{l,m_l+p_l+1}\right\}$  for $l=1,\ldots, d$  and   $\Xi_t:=\left\{
  \xi_{t,1} \leq \dots \leq \xi_{t,m_t+p_t+1}\right\}$.   We collect
the degree indexes in a vector $\boldsymbol{p} :=(\boldsymbol{p}_s,
p_t)$, where $\boldsymbol{p}_s:=(p_1,\dots,p_d)$. For the sake of simplicity,
we consider  $p_1=\dots=p_d=:p_s$ but the general case is similar. 

 In the following,   $h_s$ will denote the maximum mesh size in all spatial directions and
$h_t$ the  mesh size in the time direction. We assume that
the following quasi-uniformity condition on the knot vectors holds.  
\begin{ass}
\label{ass:local_quasi_unif}
We assume that the knot vectors
are quasi-uniform, that is, there exists  $\alpha $ such that $0<\alpha \leq 1$, independent
of $h_s$ and $h_t$, such that each non-empty knot span $( \xi_{l,i} ,
\xi_{l,i+1})$ fulfills $ \alpha h_s  \leq \xi_{l,i+1} - \xi_{l,i} \leq
h_s$, for $1\leq l \leq d$,  and  each non-empty knot span $( \xi_{t,i} ,
\xi_{t,i+1})$ fulfills $ \alpha h_t  \leq \xi_{t,i+1} - \xi_{t,i} \leq
h_t$.
\end{ass}

We denote by $\widehat{\Omega}:=(0,1)^d$ the spatial parameter
domain. We define the multivariate B-splines on
$\widehat{\Omega}\times [0,1]$ as 
\[
\widehat{B}_{ \vect{i},\vect{p}}(\vect{\eta},\tau) : =
\widehat{B}_{\vect{i_s},\vect{p_s}}(\vect{\eta}) \widehat{b}_{i_t,p_t}(\tau),
\]
 where $\widehat{B}_{\vect{i_s},\vect{p_s}}(\vect{\eta}):=\widehat{b}_{i_1,p_s}(\eta_1) \ldots \widehat{b}_{i_d,p_s}(\eta_d)$, $\vect{i_s}:=(i_1,\dots,i_d)$, $\vect{i}:=(\vect{i_s}, i_t)$ and
 $\vect{\eta} = (\eta_1, \ldots, \eta_d)$.  
The  corresponding spline space  is defined as
\[
\widehat{\S} ^{\vect{p}}_{ {h}  }  := \mathrm{span}\left\{\widehat{B}_{\vect{i}, \vect{p}} \ \middle| \ i_k = 1,..., m_k \text{ for } k=1,\dots,d; i_t=1,\dots,m_t \right\},
\]
where $h:=\max\{h_s,h_t\}$. 
We have 
$\widehat{\S} ^{\vect{p}}_{ {h}}  =\widehat{\S} ^{\vect{p}_s}_{ {h}_s} \otimes
\widehat{\S} ^{p_t}_{h_t} = \widehat{\S}^{p_s}_{h_s}\otimes \ldots \otimes
\widehat{\S}^{p_s}_{h_s}\otimes \widehat{\S} ^{p_t}_{h_t} $, where   \[ \widehat{\S} ^{\vect{p}_s}_{h_s} := \mathrm{span}\left\{ \widehat{B}_{\vect{i_s},\vect{p_s}}(\vect{\eta})\ \middle|  \ i_k = 1,..., m_k \text{ for } k=1,\dots,d  \right\}.\]

 The minimum regularity of the spline
spaces that we assume is the following.
\begin{ass}
\label{ass:disc_spaces_continuity}
We assume  that ${p}_s\geq 2$,  $\widehat{\S}_{{h}_s} ^{\vect{p}_s} \subset
C^1(\widehat{\Omega}) $, $p_t\geq 1$ and  $\widehat{\S}_{{h}_t} ^{\vect{p}_t} \subset
C^0(\widehat{\Omega})  $. 
\end{ass}

\subsection{Isogeometric spaces} 

The space
domain $\Omega\subset \mathbb{R}^d$ is given as a spline non-singular
single-patch, that is, the following conditions are fulfilled.
\begin{ass}
\label{ass: single-patch-domain}
 We assume  that   $\mathbf{F}: \widehat{\Omega} \rightarrow {\Omega}
$, with   $\mathbf{F}\in  \left [
  \widehat{\mathcal{S}}^{\vect{p_s}}_{{h}_s} \right ] ^d$ on the
closure of  $\widehat{\Omega}$.
\end{ass}
\begin{ass}
\label{ass: regularpatch-domain}
 We assume  that  $\mathbf{F}^{-1}$ has piecewise
  bounded derivatives of any order.
\end{ass}

 Let  $\vecx=(x_1,\dots,x_d):=
  \mathbf{F}(\vect{\eta})$. Given $T>0$, the space-time computational
  domain   $\Omega\times[0,T]$ is given by  the  parametrization
  $\mathbf{G}\in  \left [\widehat{\S} ^{\vect{p}}_{{h}} \right ]
  ^{d+1}$ such that 
 $\mathbf{G}:\widehat{\Omega}\times[0,1]\rightarrow \Omega\times[0,T]$
 with 
 $\mathbf{G}(\vect{\eta}, \tau):=(\mathbf{F}(\vect{\eta}), T\tau
 )=(\vecx,t)$, and 
  where $t:=T\tau$. We introduce, in the parametric domain, the space with boundary conditions
\[
\widehat{\mathcal{V}}_{h,0}:=\left\{ \widehat{v}_h\in \widehat{\mathcal{S}}^{\vect{p}}_h \ \middle| \ \widehat{v}_h = 0 \text{ on } \partial\widehat{\Omega}\times [0,1] \text{ and } \widehat{v}_h = 0 \text{ on } \widehat{\Omega}\times\{0\} \right\}.
\]
Note that 
$
\widehat{\mathcal{V}}_{h,0} = \Vhspaceref\otimes \Vhtimeref
$, 
where 
\begin{subequations}
\label{eq:basis_par}
\begin{align}
\Vhspaceref   & := \left\{ \widehat{w}_h\in \widehat{\mathcal{S}}^{\vect{p_s} }_{h_s}  \ \middle| \ \widehat{w}_h = 0 \text{ on } \partial\widehat{\Omega}  \right\}\   = \ \text{span}\left\{ \widehat{b}_{i_1,p_s}\dots\widehat{b}_{i_d,p_s} \ \middle| \ i_k = 2,\dots , m_k-1; \ k=1,\dots,d\ \right\}, \\ 
 \Vhtimeref   &:= \left\{ \widehat{w}_h\in \widehat{\mathcal{S}}^{ p_t}_{h_t} \ \middle|  \ \widehat{w}_h( 0)=0 \right\}  \   = \ \text{span}\left\{ \widehat{b}_{i_t,p_t} \ \middle| \ i_t = 2,\dots , m_t\ \right\}.  
\end{align}
\end{subequations} 
 Reordering the basis and then introducing the colexicographical
ordering of the degrees-of-freedom,   we have
\begin{align*}
  \Vhspaceref    & = \ \text{span}\left\{ \widehat{b}_{i_1,p_s}\dots\widehat{b}_{i_d,p_s} \ \middle| \ i_k = 1,\dots , n_{s,k}; \ k=1,\dots,d\ \right\}  =\text{span}\left\{ \widehat{B}_{i, \vect{p}_s} \ \middle|\ i =1,\dots , N_s   \ \right\}, \\
\  \Vhtimeref  & = \ \text{span}\left\{ \widehat{b}_{i,p_t} \ \middle| \ i = 1,\dots , n_t\ \right\}
\end{align*}
and 
\begin{equation}
\widehat{\mathcal{V}}_{h,0}=\text{span}
\left\{ \widehat{B}_{{i}, \vect{p}} \ \middle|\ i=1,\dots,N_{dof} \right\},
\label{eq:all_basis}
\end{equation}
where   we have defined 
\[ n_t:=m_t-1, \qquad n_{s,k}:= m_k-2, \qquad N_s:=\prod_{k=1}^dn_{s,k}, \qquad N_{dof}:=N_s n_t.\] 
The isogeometric space we consider is the isoparametric push-forward of $\widehat{\mathcal{V}}_{h,0}$, i.e.

\begin{equation}
\label{eq:disc_space}
\mathcal{V}_{h,0} := \text{span}\left\{  B_{i, \vect{p}}:=\widehat{B}_{i, \vect{p}}\circ \mathbf{G}^{-1} \ \middle| \ i=1,\dots , N_{dof}   \right\}.
\end{equation}
Note that  $\mathcal{V}_{h,0}$ can be written as
\[
\mathcal{V}_{h,0}=\Vhspace\otimes \Vhtime,
\]  where 
\begin{align*} 
 \Vhspace& :=\text{span}\left\{ {B}_{i, \vect{p}_s}:= \widehat{B}_{i, \vect{p}_s}\circ \mathbf{F}^{-1} \ \middle| \ i=1,\dots,N_s \right\},  \\ \Vhtime& :=\text{span}\left\{  {b}_{i,p_t}:= \widehat{b}_{i,p_t}( \cdot /T) \ \middle| \ i=1,\dots,n_t \right\}.
\end{align*}

\section{Parabolic model problem and its discretization}
\subsection{The heat equation and the regularity of its solution}
 We denote by $\partial_t$  the partial time derivative and by $\Delta $   the laplacian w.r.t. spatial variables.
  If $A$ and $B$ are Hilbert spaces,  $A\otimes B$ denotes the closure of their
tensor product (see \cite[Definition 12.3.2]{Aubin2011}).
We also  identify the spaces $H^m((0,T); H^n(\Omega))$,
$H^n(\Omega)\otimes H^m(0,T)$ and $H^{n,m}(  \Omega\times (0,T))$, 
(see \cite[Section 12.7]{Aubin2011}).  
 We denote by   $
H_{\Delta}(\Omega) $ the space  $ \left\{ z \in L^2(\Omega)\ \middle| \
  \Delta z \in L^2(\Omega) \right\}$,   and we have the following result.
 
\begin{prop}\label{prop:elliptic-regularity} Under Assumptions   \ref{ass:disc_spaces_continuity}--\ref{ass: regularpatch-domain}, 
  there exists a constant $C_{\Delta}> 0$, depending only on the space parametrization $\mathbf{F}$, such that 
\begin{equation}\label{eq:elliptic-reg}
\| z\|_{H^2(\Omega)}^2 \leq C_{\Delta}\|\Delta z\|_{L^2(\Omega)}^2
\qquad \forall z \in  H^1_0(\Omega)\cap H^2(\Omega).
\end{equation}
\end{prop}
\begin{proof}
  From  Assumptions \ref{ass:disc_spaces_continuity}--\ref{ass:  regularpatch-domain},  
  $\Omega$ has a  piecewise smooth  boundary with  bounded
  curvature  and non-null interior angles (see the definition in \cite[Chapitre III, pag. 161]{Ol1968}). Then, we can use \cite[Chapitre III, Lemme 11.1]{Ol1968}.  
\end{proof}

 We  define the  space  
 \begin{align*}
\V_{ 0} := & \left\{   v \in \left[ \left  (H^1_0(\Omega)\cap
  H^2(\Omega) \right )\otimes L^2(0,T)\right] \cap \left[  L^2(\Omega)\otimes H^1(0,T) \right]\  \text{ s.t. }   \ \  v=0   \text{ on } \Omega\times\{0\}  \right\},
\end{align*} 
 endowed with the norm 
\begin{equation}
\label{eq:norm_V0}
\|v\|_{\V_0}^2:= \int_{0}^T\|\Delta v (\cdot,t)  \|_{L^2(\Omega)}^2  \,\dt +  \int_{0}^T \|\partial_t v (\cdot,t) \|_{L^2(\Omega)}^2\,\dt.
\end{equation} 
Thanks to Proposition \ref{prop:elliptic-regularity}, $\V_0$ is a
Hilbert space and the $\|\cdot\|_{\mathcal{V}_0}$-norm is equivalent to 
\begin{equation}
\label{eq:norm_|||}
 \vert\vert\vert v \vert\vert\vert ^2:= \|v\|_{ H^2(\Omega) \otimes
  L^2(0,T)}^2+\|v\|_{L^2(\Omega)\otimes H^1(0,T)}^2.
\end{equation} 
 
 { 
Our model problem is  the    heat  equation,  with initial and homogeneous
boundary conditions: we seek for a solution $u $ such that 

\begin{equation}
\begin{cases}
\label{eq:heat_eq}
		\partial_t u - \Delta u \ = \  f  \quad &\mbox{in }\ \ \ \hspace{1mm} \Omega \times (0, T),  \\
		 u  \ =\  0  \quad \quad & \mbox{on }\  \hspace{0.7mm}   \partial \Omega   \times  (0,T), \\
	 	 u \ =\  0\quad & \mbox{in }\ \  \ \ \Omega  \times \lbrace 0  \rbrace. 
\end{cases}
\end{equation}
 with $f\in
 L^2(\Omega \times (0,T))$. Before proving the  theorem assessing the regularity of the solution $u$ of \eqref{eq:heat_eq}, we need the following lemma.
 \begin{lemma}
 \label{lemma:ellip_reg_domain}
  Let  Assumptions \ref{ass: single-patch-domain}--\ref{ass:  regularpatch-domain}   hold and let $r\in L^2(\Omega)$.  Then,  there exists a unique weak solution $z\in H^2(\Omega)$ to the Poisson problem 
\begin{equation}
\begin{cases}
\label{eq:poisson}
  - \Delta  {z} \ = \  {r}  \quad & \mbox{in }\ \ \ \hspace{1mm}   \Omega, \\
\hspace{0.58cm} {z}  \ = \  0  \quad & \mbox{on }\  \hspace{0.2cm}  \partial \Omega.    
\end{cases}
\end{equation}
Moreover, there exists a constant $C$ depending only on  $\mathbf{F}$ such that
\begin{equation}
\label{eq:estimate_regul}
\|z\|_{H^2(\Omega)}\leq C\|r \|_{L^2(\Omega)}.
\end{equation}
 \end{lemma}
 
 \begin{proof}
 We recall that $z$ is a weak solution of \eqref{eq:poisson} if $z\in H^1_0(\Omega)$  and if $\int_{\Omega}\nabla z \cdot \nabla q\, \d\Omega = \int_{\Omega} rq\,\d\Omega \ \forall q\in H^1_0(\Omega)$.
 Then, we have that  $z\in H^1_0(\Omega)$ is a weak solution of \eqref{eq:poisson} if and only if   $w:=z\circ\mathbf{F}\in H^1_0(\widehat{\Omega})$ is a weak solution of
 \begin{equation}
\begin{cases}
\label{eq:poisson2}
  - \nabla \cdot\left(  \vect{R}{\ \nabla w}\right) \    = \  {g}  \quad & \mbox{in }\ \ \ \hspace{1mm}   \widehat{\Omega}, \\
\hspace{1.09 cm}\qquad {w}   \   = \  0  \quad & \mbox{on }\  \hspace{0.2cm}  \partial \widehat{\Omega},   
\end{cases}
\end{equation}
 where $g:=|\text{det}(J_{\mathbf{F}})|r\circ\mathbf{F}$ and      $\vect{R}:=J_{\mathbf{F}}^{-1}J_{\mathbf{F}}^{-T}|\text{det}(J_{\mathbf{F}})|$.  
 Thanks to Assumptions  \ref{ass: single-patch-domain}--\ref{ass:  regularpatch-domain},   we have that  
$\mathbf{F}:\widehat{\Omega}\rightarrow\Omega$ fulfils
$\mathbf{F}\in C^{1,1}$ on the closure of $  {\widehat{\Omega}} $   and 
$\mathbf{F}^{-1}\in C^{1,1} (\overline{{\Omega}} )$. 
 Therefore, we have that the entries of the
 matrix $\vect{R}$ are Lipschitz continuous and we can apply    \cite[Theorem 3.2.1.2]{Grisvard201}  to  conclude that there exists a unique solution $w\in H^2(\widehat{\Omega})$ of problem \eqref{eq:poisson2}. Thanks to \cite[Lemma 11.1]{Ol1968} we also have 
\begin{align*}
\|w\|_{H^2(\widehat{\Omega})}^2 & \leq {c_1} \left( \|\nabla \cdot\left(  \vect{R}{\  \nabla w}\right)\|_{L^2(\widehat{\Omega})}^2 + \|w\|_{L^2(\widehat{\Omega})}^2 \right)\\& \leq  { c_2} \| \nabla\cdot ( \vect{R} {\ \nabla w}) \|_{L^2(\widehat{\Omega})}^2 \\ 
& = { c_2} \|g\|_{L^2(\widehat{\Omega})}^2,
\end{align*}
where {$c_1$ and $c_2$ are constants} depending
only on  $ \vect{R}$, that is, on  $\mathbf{F}$ and its inverse.  
Finally, we conclude
\[
\|z\|_{H^2(\Omega)}\leq C_1\|w\|_{H^2(\widehat{\Omega})}\leq C_2 \|g\|_{L^2(\widehat{\Omega})}\leq C \|r\|_{L^2(\Omega)},
\]
where the constants $C_1, C_2$ and $C$ depend only on  $\mathbf{F}$.
 \end{proof}
 
 \begin{tw}\label{teo:parabolic-regularity} 
Let $f\in L^2( \Omega \times (0,T) )$ and let Assumptions   \ref{ass:local_quasi_unif}-\ref{ass:  regularpatch-domain}   hold. Then there exists a unique weak
solution (as defined in  \cite[Chapter 7]{Evans2010book})     
 $u\in \left(  H^2(\Omega) \otimes L^2(0,T)
 \right)\cap \left( L^2(\Omega)\otimes H^1(0,T) \right) \cap
\left(H^1_0(\Omega)\otimes  L^{\infty}(0,T)\right)$    of 
\eqref{eq:heat_eq}. We also have  
\[
  \|u\|_{ H^2(\Omega)\otimes L^2(0,T)}+ \|u\|_{L^2(\Omega)\otimes  H^1(0,T)} +\|u\|_{ H^1_0(\Omega)\otimes L^{\infty}(0,T)} \leq {C} \|f\|_{L^2(\Omega\times (0,T))},
\]
where ${C}$ is a constant depending only on  $\mathbf{F}$.
\end{tw}
\begin{proof}

Following the same arguments of step 1 and step 2 of the proof of \cite[Section 7, Theorem 5]{Evans2010book}, we conclude that $u\in \left(H^1_0(\Omega) \otimes   L^{\infty}(0,T)\right) \cap\left( L^2(\Omega)\otimes H^1(0,T)\right)$ and that 
\begin{equation}
\label{eq:first_est}
 \|u\|_{ L^2(\Omega) \otimes H^1(0,T)} + \|u \|_{ H^1_0(\Omega) \otimes L^{\infty}(0,T) }   \leq D_1 \|f\|_{L^2(\Omega\times (0,T))},
\end{equation}
where $D_1$ is a constant depending only on  $\mathbf{F}$.

  We write for a.e. $t$  
\[  
\int_{\Omega}\nabla u(\vecx,t)\cdot  \nabla v(\vecx)\,\d\Omega = \int_{\Omega}r(\vecx,t)  \ v(\vecx)\,\d\Omega \quad \forall v \in H^1_0(\Omega), 
\]
where $ r:=f-\partial_t u\in L^2(\Omega\times (0,T))$ and in particular
$r(\cdot ,t )\in L^2(\Omega)$ for a.e. $t$. Therefore, thanks to Lemma
\ref{lemma:ellip_reg_domain}, we conclude that $u(\cdot,t)\in
H^2(\Omega)$ for a.e. $t$ and thus    $u \in H^2(\Omega)\otimes L^2(0,T)$: indeed, integrating in time, \eqref{eq:estimate_regul} and \eqref{eq:first_est} yield  to the following estimate  
\begin{align*}
\|u\|^2_{ H^2(\Omega)\otimes L^2(0,T)}& \leq C^2 \|r\|^2_{L^2(\Omega\times (0,T))} \\ & \leq C^2(\|f\|^2_{L^2(\Omega\times(0,T))} + \|u\|^2_{L^2(\Omega) \otimes  H^1(0,T)})\\
& \leq  {D}_2^2\|f\|^2_{L^2(\Omega\times(0,T))},
\end{align*}
where ${ {D}_2^2}:=C^2+D_1^2$. This concludes the proof.
\end{proof}
  
More generally, non-homogeneous initial and boundary
conditions are allowed. For example,   if $ u \ =\  u_0$ in $ \Omega
\times \lbrace 0  \rbrace$, with    $u_0\in  H^1_0(\Omega)$,  we 
lift\footnote{  We can use  the same argument as in
  Theorem~\ref{teo:parabolic-regularity} that is,  the  proof of \cite[Section 7, Theorem 5]{Evans2010book},
  where step 3 therein uses  the  elliptic regularity  property
  which is  given, in our case, by
  Lemma \ref{lemma:ellip_reg_domain}.} $ u_0$ to $\widetilde{u}_0\in (H^1_0(\Omega)\cap H^2(\Omega)) \otimes L^2(0,T) \cap
 L^2(\Omega)\otimes H^1(0,T)$. Then $\widetilde{u}   ={u} -
\widetilde{u}_0  \in \V_0$} is the solution of   
\begin{equation}
\begin{cases}
\label{eq:heat_eq-zero-initial-condition}
		\partial_t \widetilde{u} - \Delta \widetilde{u} \ = \  \widetilde{f}  \quad &\mbox{in }\ \ \ \hspace{1mm}   \Omega \times (0, T), \\
		 \widetilde{u}  \ =\  0  \quad & \mbox{on }\  \hspace{0.7mm}  \partial \Omega   \times  (0,T), \\
		 \widetilde{u} \ =\   0 \quad &\mbox{in }\ \ \ \ \Omega  \times \lbrace 0  \rbrace,  
\end{cases}
\end{equation}
where $\widetilde{f}:=f - 	\partial_t\widetilde{u}_0 + \Delta \widetilde{u}_0$.
For a detailed description of the variational
 formulation of   problems \eqref{eq:heat_eq}--\eqref{eq:heat_eq-zero-initial-condition}  and their well-posedness
 see, for example,   \cite{Evans2010book,Schwab2009}.

\subsection{ Least-squares variational formulation}
\label{sec:variational-formulation}
{ We consider  the following  variational formulation for the system \eqref{eq:heat_eq}:}
  find $u\in \mathcal{V}_{0}$ such that 
\begin{equation}
\label{eq:min_prob}
u =\underset{v\in \mathcal{V}_{0}}{\text{arg min}}\tfrac12 \left\| \partial_tv -\Delta v-f  \right\|^2_{L^2(\Omega \times (0,T))}.
\end{equation}
 Its  Euler-Lagrange equation is
\begin{equation}
\label{eq:cont_problem}
\mathcal{A}(u,v) = \mathcal{F}(v) \quad \forall v\in \mathcal{V}_{0},
\end{equation}
where the bilinear form $\mathcal{A}(\cdot,\cdot)$ and the linear form $\mathcal{F}(\cdot)$  are defined as
\begin{equation}
\label{eq:a-form}
 \mathcal{A}(v,w) := \int_{0}^T\int_{\Omega}  \left( \partial_t{v} \, \partial_t w +  \Delta v\, \Delta w -  \partial_t{v} \, \Delta w   - \Delta v\, \partial_t w \right)\,\d\Omega\,   \dt, 
 \end{equation}
 \[\mathcal{F}(w)  := \int_0^T\int_{\Omega}f\, ( \partial_t w-\Delta w)\,\d\Omega\,\dt.\]
For an equivalent way of writing the minimization problem
\eqref{eq:min_prob}, we refer to Appendix \ref{App}. The variational formulation \eqref{eq:cont_problem}  is well-posed,
thanks to the following Lemmas \ref{lemma:cont}--\ref{lemma:cont-F}
and Proposition \ref{prop:LM-continuous}.
 
\begin{lemma}
\label{lemma:cont}
The bilinear form $\mathcal{A}(\cdot,\cdot)$ is continuous in $\V_0$.  Particularly, it holds
\[
|\mathcal{A}(v,w)|\leq 2\|v\|_{\V_0} \|w\|_{\V_0} \quad \forall v, w\in \V_0.
\]
\end{lemma}
\begin{proof}
Given $v,w\in \V_0$, by   Cauchy-Schwarz inequality  
\begin{align*}
|\mathcal{A}(v,w)| & \leq\|v\|_{\V_0}\|w\|_{\V_0}   + \int_0^T \int_{\Omega} \left| \partial_t{v}\,  \Delta{w}  \right| \,\d\Omega    \,\dt + \int_0^T \int_{\Omega} \left| \Delta v\, \partial_t{w}\right| \,\d\Omega    \,\dt \\
& \leq \|v\|_{\V_0}\|w\|_{\V_0}    + \left[ \int_0^T \left(\|\partial_t{v} (\cdot,t) \|_{L^2(\Omega)}^2 + \|\Delta v (\cdot,t) \|^2_{L^2(\Omega)}\right)\,\dt\right]^{1/2}   \\
& \quad \cdot  \left[ \int_0^T \left(\|\partial_t{w} (\cdot,t) \|_{L^2(\Omega)}^2 + \|\Delta w (\cdot,t) \|^2_{L^2(\Omega)}\right)\,\dt\right]^{1/2} \\
& \leq 2\|v\|_{\V_0}\|w\|_{\V_0},
\end{align*}
which concludes the proof.
\end{proof}

\begin{lemma}  
\label{lemma:coer}
The bilinear form $\mathcal{A}(\cdot,\cdot)$ is $\V_{0}$-elliptic.  In particular, it holds
\[
\mathcal{A}(v,v) \geq  \|v\|^2_{\V_0} \quad \forall v\in \V_{0}.
\]
 
\end{lemma} 
\begin{proof}
Let $v \in \V_{0}$. Thanks to    \cite[Lemme 3.3]{Brezis1973},   we can write
\[
 -2\int_0^T \int_{\Omega} \partial_tv\,\Delta v  \,\d\Omega\,\dt 
  =  \int_{\Omega} |\nabla v(\vecx, T)|^2 \,\d\Omega - \int_{\Omega} |\nabla v(\vecx, 0)|^2 \,\d\Omega , 
\]
where  $\nabla :=[\partial_{x_1},\dots,\partial_{x_d}]^T$  denotes the  gradient w.r.t. spatial variables $x_1,\dots,x_d$.
In particular, as $\nabla v(\vecx, 0)=0$, we have that $ \forall v \in \V_0$
 \begin{align*} \mathcal{A}(v,v)& = \int_{0}^T  \| \partial_tv(\cdot,t) \|^2_{L^2(\Omega)} \, \dt + \int_{0}^T \|\Delta v(\cdot,t) \|^2_{L^2(\Omega)} \,\dt + \int_{\Omega} |\nabla  v(\vecx, T)|^2\,\d\Omega \\
 & \geq \|v\|^2_{\V_0},
 \end{align*}
which concludes the proof.
\end{proof}
\begin{lemma}\label{lemma:cont-F}

The linear form $\mathcal{F}(\cdot)$ is continuous in $\V_0$. In particular it holds
\[
\mathcal{F}(v)\leq  \sqrt{2}\|f\|_{L^2(\Omega\times (0,T)) }\|v\|_{\V_0} \quad \forall v\in \V_0.
\]
\end{lemma}
\begin{proof}
Given $v\in\V_0$, by Cauchy-Schwarz inequality we get
 \begin{align*} 
|\mathcal{F}(v)|& \leq  \|f\|_{L^2(\Omega\times (0,T)) }\left(\int_0^T  \| \partial_t{v}(\cdot,t) 
    - \Delta{v}(\cdot,t) \|_{L^2(\Omega)}^2\ \dt\right)^{1/2}\\
 & \leq   \sqrt{2}\|f\|_{L^2(\Omega\times(0,T)) }\left(  \int_0^T \| \partial_t{v}(\cdot,t) \|_{L^2(\Omega)}^2\ \dt
    +\int_0^T\| \Delta{v}(\cdot,t) \|_{L^2(\Omega)}^2 \  \dt \right)^{1/2}  \\
     & =  \sqrt{2} \|f\|_{L^2(\Omega\times(0,T)) } \|v\|_{\V_0},
\end{align*} 
which concludes the proof.
\end{proof}

 \begin{prop}\label{prop:LM-continuous}
 Under Assumptions  \ref{ass:disc_spaces_continuity}--\ref{ass:  regularpatch-domain},    the minimization problem \eqref{eq:min_prob} and the variational problem \eqref{eq:cont_problem} are equivalent and they admit  a  unique solution $u\in \V_{0}$. 
 \end{prop}
 \begin{proof}
 The proof follows using Lemmas \ref{lemma:cont}--\ref{lemma:cont-F} and the Lax-Milgram theorem.
 \end{proof}

\subsection{Least-squares approximation}
\label{sec:least_squares}
 
Thanks to Assumption \ref{ass:disc_spaces_continuity},  we have  
\begin{equation}
  \label{eq:discrete-space-conformity}
  \V_{h,0}\subset(H^1_0(\Omega)\cap H^2(\Omega))\otimes H^1(0,T)\subset\V_0.
\end{equation}
 Therefore, we consider  a Galerkin method for
 \eqref{eq:cont_problem}, that is, the  least-squares approximation of the
 system \eqref{eq:heat_eq}: find $u_h\in \V_{h,0}$ such that 
\begin{equation}
\label{eq:discrete_min_problem}
u_h =\underset{v_h\in \mathcal{V}_{h,0}}{\text{arg min}} \tfrac12 \left\| \partial_tv_h -\Delta v_h-f  \right\|^2_{L^2(\Omega \times (0,T))}.
\end{equation}
   Its   Euler-Lagrange  equation is  
\begin{equation}
\label{eq:discrete_problem}
\mathcal{A}(u_h,v_h) = \mathcal{F}(v_h) \quad \forall v_h\in \mathcal{V}_{h,0}.
\end{equation} 
 
 Well-posedness and quasi-optimality  follow from standard arguments.
 \begin{prop}\label{prop:LM-discrete}
The minimization problem \eqref{eq:discrete_min_problem} and the variational problem \eqref{eq:discrete_problem} are equivalent and they admit  a  unique solution  $u_h\in\V_{h,0}$.   It also holds: 
 \begin{equation}
   \label{eq:quasi-optimality}
   \|u-u_h\|_{\V_0}\leq \sqrt{2}  \underset{v_h\in \mathcal{V}_{h,0}}{ {\inf}} \|u-v_h\|_{\V_0}.
 \end{equation}
 \end{prop}
 \begin{proof}
 The proof of the equivalence and of the existence and uniqueness of a solution follow by using  Lemmas \ref{lemma:cont}--\ref{lemma:cont-F} and
 the   Lax-Milgram theorem, while the proof of   \eqref{eq:quasi-optimality} is a consequence of
 the  C\' ea Lemma and  the symmetry
 of the bilinear form $\mathcal{A}$.
 \end{proof}

 The following result states the convergence of our method.
  
 \begin{tw} \label{teo:convergence}
 Under Assumptions  \ref{ass:disc_spaces_continuity}--\ref{ass:  regularpatch-domain},   we have  $\lim_{h\rightarrow 0} \|u-u_h\|_{\V_0} = 0$.
 \end{tw}

 \begin{proof}
To prove the theorem, we show that
\begin{equation}
  \label{eq:density}
  \lim_{h\rightarrow 0} \underset{v_h\in
  \mathcal{V}_{h,0}}{ {\inf}} \|u-v_h\|_{\V_0} = 0,
\end{equation}
and then use \eqref{eq:quasi-optimality}.

Given $u \in \V_0$, let $\widehat{u} = u \circ \mathbf{G}^{-1} $ be its
pullback. Since  $\mathbf{G}$ and $\mathbf{G}^{-1}$ are both of class
$W^{2,\infty}$  and since the $\V_0$-norm \eqref{eq:norm_V0} is
equivalent to the $|||\cdot|||$-norm \eqref{eq:norm_|||},  the pullback is an isomorphism between  $\V_0$ and
 \begin{align*}
\widehat{\V}_{ 0} = & \left\{ v \in \left[ \left  ( H^2  (  \widehat{\Omega}   )\cap
  H^1_0 (\widehat{\Omega} )\right )\otimes L^2(0,1)\right] \cap
\left[  L^2  ( \widehat{\Omega}  )\otimes H^1(0,1) \right]\
 \text{ s.t. }   \ v=0   \text{ on }   \widehat{\Omega}  \times\{0\}  \right\}, 
\end{align*}
   endowed with the norm
 
\[
\|v\|_{\widehat{\V}_0}^2:=  
\int_{0}^1\|\Delta v (\cdot,\tau) \|_{L^2(\widehat\Omega)}^2  \,\d\tau +  \int_{0}^1 \|\partial_\tau v(\cdot,\tau) \|_{L^2(\widehat{\Omega})}^2\,\d\tau.  
\]
 
 Then, by using Lemma \ref{lemma:smooth-approx-2} reported in the Appendix \ref{App-smooth-approximation},   we can approximate, as
close as we want,   $\widehat
u  \in \widehat{\V}_{ 0}  $ by a  
 smooth  function fulfilling the same
boundary conditions of $\widehat u$, and then by a spline in
  $ \widehat{\V}_{h,0}$ (see \eqref{eq:all_basis}),  on a fine enough mesh. This implies  \eqref{eq:density}.
 \end{proof}
 
 \subsection{A priori error analysis} 
We investigate in this section the approximation properties of the  isogeometric space $\V_{h,0}$
under $h$-refinement. 
\begin{prop}
\label{prop:projec_norm}
Let $q_s$ and $q_t$ be two integers such that $2\leq q_s \leq p_s+1$ and $1 \leq q_t\leq p_t+1$.  Under   Assumption \ref{ass:local_quasi_unif}, there
exists a projection $\Pi_h: \V_0 \cap \left( H^{q_s}(\Omega)\otimes  
  H^1(0,T)\right)\cap \left(     H^2(\Omega)\otimes
  H^{q_t}(0,T)\right) \rightarrow \V_{h,0}$ such that 
\begin{equation}
  \label{eq:anisotropic-quasi-interpolant-local-error}
  \left\| v-\Pi_h v \right\|_{\V_0} \leq C \left(h_s^{q_s-2} \|v
    \|_{H^{q_s}(\Omega) \kron   H^1 (0,T) \B }  +  h_t^{q_t-1} \|v
    \|_{  H^2(\Omega) \B\kron H^{q_t}(0,T)} \right)  
\end{equation}
where the constant $C$ depends on $p_s,\ p_t,\ \alpha$ and the  parametrization $\mathbf{G}$.
\end{prop}
\begin{proof}
  The result follows from the anisotropic approximation estimates that
  are developed in \cite{Da2012}. We remark that \cite{Da2012}
  states  its error analysis for  $2$ dimensions, but the results
  therein straightforwardly  generalize to higher dimension. We give
  an overview of the proof, for the sake of completeness. 
  
As space and time coordinates in $\Omega \times [0,T]$ are orthogonal, 
the parametric coordinate (tangent) vectors  are
 \begin{align*} 
\vect{g}_i(\vecx)&:= \partial_{\eta_i}\mathbf{G}\circ\mathbf{G}^{-1}(\vecx, t) =\begin{bmatrix}
\partial_{\eta_i}\mathbf{F}\circ \mathbf{F}^{-1}(\vecx)\\
0
\end{bmatrix}\in \mathbb{R}^d\times\{0\}\subset\mathbb{R}^{d+1} \quad \text{ for } i=1,\dots,d,\\
\vect{g}_t(t)&:= \partial_{\tau}\mathbf{G}\circ\mathbf{G}^{-1}(\vecx, t) = \begin{bmatrix}
0\\
\vdots\\
0\\
T
\end{bmatrix}\in\mathbb{R}^{d+1}.
\end{align*} 
Then, given $v\in\V_0$, the directional derivatives w.r.t. $\vect{g}_i$ and $\vect{g}_t$ that are used in \cite[Section 5]{Da2012}, become
\[
\begin{bmatrix}
\frac{\partial v(\vecx, t)}{\partial \vect{g}_1}\\
\vdots\\
\frac{\partial v(\vecx, t)}{\partial \vect{g}_d} 
\end{bmatrix}  = \left(J_{\mathbf{F}}\circ\mathbf{F}^{-1}(\vecx)\right)^T\nabla_{\vecx}v(\vecx,t) , \qquad \qquad
\frac{\partial v}{\partial \vect{g}_t}(\vecx, t)   = T \, \partial_t v(\vecx, t).
\]
Higher-order directional derivatives can be defined similarly, as in \cite[Section 5]{Da2012}. We also have that  
\begin{subequations}
\label{eq:ineq_deriv}
\begin{align} 
\left\|\frac{\partial}{\partial \vect{g}_{i_1}}\left(\dots \frac{\partial v}{\partial\vect{g}_{i_k}}\right) \right\|_{L^2(\Omega\times(0,T))}& \leq C \|v\|_{H^k(\Omega)\otimes L^2(0,T)},\\
\left\|  \frac{\partial^kv}{\partial \vect{g}_t^k} \right\|_{L^2(\Omega\times (0,T))}& \leq C \|v\|_{L^2(\Omega)\otimes H^k(0,T)},
\end{align}
\end{subequations}
for a suitable constant $C$, $k\geq 0$ and $i_j\in\{1,\dots,d\}$, $j=1,\dots,k$.
Therefore,  \cite[Theorem 5.1]{Da2012} generalized to $d+1$ dimensions
gives the existence of a projection   $\Pi_h$ on the space $\V_{h,0}
$    such that
 \begin{align*} 
 \left\| v-\Pi_h v \right\|_{H^{2}(\Omega)\otimes L^2(0,T) } & \leq C \left(h_s^{q_s-2} \|v \|_{H^{q_s}(\Omega) \kron L^2(0,T)}  +   h_t^{q_t-1} \|v\|_{H^{2}(\Omega) \kron H^{q_t-1}(0,T)} \right), \\
     \left\| v-\Pi_h v \right\|_{ L^2(\Omega)\otimes H^1(0,T)} & \leq
                                                                 C  
                                                                 \left(h_s^{q_s-2}
                                                                 \|v
                                                                 \|_{H^{q_s-2}(\Omega)
                                                                 \kron
                                                                 H^1(0,T)}
                                                                 \B+  h_t^{q_t-1} \|v\|_{L^2(\Omega) \kron H^{q_t}(0,T)} \right) , 
\end{align*}
with $C$ depending only on $p_s,\ p_t,\ \alpha$ and the space parametrization $\mathbf{G}$.
Squaring and summing the two inequalities above,  using
\eqref{eq:ineq_deriv} and that 
$$\int_{0}^T\left\| \Delta (v-\Pi_h v)(\cdot,t) \right\|^2_{L^2(\Omega)}\dt \leq \left\| v-\Pi_h v \right\|^2_{H^{2}(\Omega)\otimes L^2(0,T) },$$
  leads to
 \begin{align*} 
 \left\| v-\Pi_h v \right\|_{\V_0} \leq \ & C h_s^{q_s-2}
                                           \left(\|v
                                           \|_{H^{q_s}(\Omega) \kron
                                           L^2(0,T)}   +  \|v
                                                                 \|_{H^{q_s-2}(\Omega)
                                                                 \kron
                                                                 H^1(0,T)}
                                           \right) \\ & +  C
                                                         h_t^{q_t-1}\left(\|v\|_{H^{2}(\Omega) \kron H^{q_t-1}(0,T)} 
                                                         +  \|v\|_{L^2(\Omega) \kron H^{q_t}(0,T)} \right) , 
\end{align*} 
and  finally \eqref{eq:anisotropic-quasi-interpolant-local-error} by
the obvious upperbound of the right-hand-side norms.
 \end{proof}

  As a direct corollary of Proposition \ref{prop:LM-discrete} and
\ref{prop:projec_norm}, we can now state the a-priori error estimate for the least-squares method. 
\begin{tw}
\label{teo:a-priori}
Let $q_s$ and $q_t$ be two integers such that $q_s\geq 2$ and $  q_t\geq 1$.
If $u\in \V_0 \cap \left( H^{q_s}(\Omega)\otimes  
  H^1(0,T)\right)\cap \left(     H^2(\Omega)\otimes H^{q_t}(0,T)\right)$ is the solution of \eqref{eq:heat_eq} and $u_h\in \V_{h,0}$ is the solution of \eqref{eq:discrete_problem}, then \begin{equation}
\label{eq:a_priori_estimate}
\|u-u_h\|_{\V_0} \leq  {C}( h_s^{k_s-2}\|u\|_{H^{k_s}(\Omega)\otimes  
  H^1(0,T)} + h_t^{k_t-1}\|u\|_{    H^2 (\Omega)\B \otimes H^{k_t}(0,T)} )
\end{equation}
where $k_s:=\min\{q_s, p_s+1\}$, $k_t:=\min\{q_t, p_t+1\}$,  $C$ is a constant that depends only on  $p_s$, $p_t$, $\alpha$ and the  parametrization $\mathbf{G}$.
\end{tw}

\section{Linear solver}
In this section we analyze solving strategies for  the least-squares method \eqref{eq:discrete_problem} and   we present a suitable preconditioner.

We recall that the Kronecker product between two matrices $\vect{C}\in\mathbb{R}^{n_1\times n_1}$ and $\vect{D}\in\mathbb{R}^{n_2\times n_2}$ is defined as
\[
\vect{C}\otimes \vect{D}:=\begin{bmatrix}
[\vect{C}]_{1,1}\vect{D}  & \dots& [\vect{C}]_{1,n_1}\vect{D}\\
\vdots& \ddots &\vdots\\
[\vect{C}]_{n_1, 1}\vect{D}& \dots & [\vect{C}]_{n_1, n_1}\vect{D}
\end{bmatrix}\in \mathbb{R}^{n_1n_2\times n_1 n_2},
\]
where $[\vect{C}]_{i,j}$ denotes the  $ij$-th entry of the matrix $\vect{C}$.
We will use the following properties (see  \cite{Kolda2009}):
\begin{itemize}
\item it holds \begin{equation}
\label{eq:kron_transp}
(\vect{C}\otimes \vect{D})^T=\vect{C}^T\otimes\vect{D}^T;
\end{equation} 
{ \item if $\vect{C}$, $\vect{D}$, $\vect{E}$ and $\vect{F}$ are matrices and there exist the products $\vect{C}\vect{E}$ and $\vect{D}\vect{F}$, it holds
\begin{equation}
\label{eq:kron_prod}
(\vect{C}\otimes \vect{D}) \cdot (\vect{E}\otimes \vect{F}) = (\vect{CE}) \otimes (\vect{DF});
\end{equation} }
\item if $\vect{C}$ and $\vect{D}$ are non-singular, then
\begin{equation}
\label{eq:kron_inv}
(\vect{C}\otimes \vect{D})^{-1}=\vect{C}^{-1}\otimes\vect{D}^{-1};
\end{equation}  
 \item if $\vect{X}\in \mathbb{R}^{n_1\times n_2}$ then
 \begin{equation}
 \label{eq:kron_vec}
 (\vect{C}\otimes \vect{D})\text{vec}(\vect{X})=\text{vec}(\vect{D}\vect{X}\vect{C}^T)
 \end{equation}
 where the  vectorization ``vec" operator applied to a matrix stacks its columns  in a vector as
\[[ \text{vec}(\vect{X})]_{i_1+(i_2-1)n_1}=[\vect{X}]_{i_1,i_2} \qquad i_j=1,\dots,n_j \text{ and } j=1,2.\]
\end{itemize}
  
  We  recall that, for $m=1,\dots,d+1,$ the $m$-mode product of a tensor $\mathfrak{X}\in\mathbb{R}^{n_1\times\dots\times n_{d+1}}$ with a matrix $\vect{W}\in\mathbb{R}^{w\times n_m}$   is a tensor of size $n_1\times\dots\times n_{m-1}\times w \times n_{m+1}\times \dots n_{d+1}$ whose elements are defined as
\[\left[ \mathfrak{X}\times_m \vect{W} \right]_{i_1, \dots, i_{d+1}} = \sum_{j=1}^{n_m} [\mathfrak{X}]_{i_1,,\dots, i_{m-1},j,i_{m+1}\dots,i_{d+1}}[\vect{W}]_{i_m,j }.\]
An important property, that represents the generalization to the $(d+1)$-dimensional case of \eqref{eq:kron_vec}, is  the following one: if $\vect{W}_i\in\mathbb{R}^{w_i\times n_i}$ for $i=1,\dots, d+1$, then
\begin{equation}
\label{eq:kron_vec_multi}
\left(\vect{W}_{d+1}\otimes\dots\otimes \vect{W}_1\right)\text{vec}\left(\mathfrak{X}\right)=\text{vec}\left(\mathfrak{X}\times_1 \vect{W}_1 \dots \times_{d+1}\vect{W}_{d+1} \right)
\end{equation} 
 where the vectorization operator ``vec" applied to a tensor stacks its entries  into a column vector as
 \[ [\text{vec}(\mathfrak{X})]_{j}=[\mathfrak{X}]_{i_1,\dots,i_{d+1}}\qquad i_l=1,\dots,n_{l} \text{ and } l=1,\dots,d+1,\] 
 where $j=i_1+\sum_{k=2}^{d+1}\left[(i_k-1)\Pi_{l=1}^{k-1}n_l\right]$.
 
\subsection{Discrete system}
Before introducing the discrete system, we rewrite the bilinear form $\mathcal{A}(\cdot, \cdot)$ in an equivalent way, through  the following Lemma. 
\begin{lemma}
 The bilinear form $\mathcal{A}(\cdot, \cdot)$  can be written as 
\begin{align}
\label{eq:equivalent_A} 
\mathcal{A}(v_h,w_h)& = \int_{0}^T  \int_{\Omega} \partial_tv_h\, \partial_tw_h\,\d\Omega \, \dt  + \int_{0}^T \int_{\Omega}\Delta v_h\, \Delta w_h\,\d\Omega \,\dt   + \int_{\Omega} \nabla  v_h(\vecx, T)\cdot \nabla  w_h(\vecx, T)  \,\d\Omega  
\end{align}
 for all  $  v_h,w_h\in\V_{h,0}.$
\end{lemma}
\begin{proof}
Let $v_h, w_h\in \V_{h,0}$. First note that $\partial_t
v_h, \partial_t w_h\in \left(H_0^1(\Omega)\cap
  H^2(\Omega)\right)\otimes L^2(0,T)$, from
\eqref{eq:discrete-space-conformity}, and   $\partial_tv_h=\partial_tw_h=0$ on $\partial \Omega\times[0,T]$. Using Green formula and integrating by parts  yields to
\begin{align*} 
   -\int_0^T \int_{\Omega}  \left(\partial_tv_h\,\Delta w_h \right.   \left.+ \partial_tw_h\Delta v_h\right)\,\d\Omega\,\dt 
  & =\,  -\int_{0}^T  \int_{\partial\Omega}  \left( \partial_tv_h\nabla  w_h +  \partial_tw_h\nabla v_h\right)\cdot \vect{\nu} \,\d\Omega\,\dt \\
&\ \ \ \  + \int_0^T \int_{\Omega}\left[ \nabla ( \partial_tv_h) \cdot \nabla  w_h + \nabla ( \partial_tw_h) \cdot \nabla v_h\right]\,\d\Omega\,\dt \\ &  =   \int_0^T\left[ \partial_t\left(\int_{\Omega}   \nabla  {v}_h \cdot \nabla w_h  \,\d\Omega\right) \right]  \dt \\
& =   \int_{\Omega}\left[  \nabla v_h(\vecx, T)\cdot \nabla w_h(\vecx, T)-  \nabla  v_h(\vecx, 0) \cdot \nabla  w_h(\vecx, 0) \right]\,\d\Omega  \\
& =  \int_{\Omega} \nabla v_h(\vecx, T)\cdot \nabla w_h(\vecx, T)  \,\d\Omega,
\end{align*}
where $\vect{\nu}\in\mathbb{R}^{d}$ is the external normal unit vector
to $\partial\Omega$. Then \eqref{eq:equivalent_A} follows.
\end{proof}

 \begin{rmk}
 Note that the identity \eqref{eq:equivalent_A} holds also in the continuous setting (see Appendix \ref{App}).
 \end{rmk}

 After the introduction of the basis \eqref{eq:disc_space} for $\V_{h,0}$, the linear system associated to   \eqref{eq:discrete_problem} is
\[ \vect{A}\mathbf{u}=\vect{F}\]
where $[\vect{A}]_{i,j}:=\mathcal{A}(B_{i, \vect{p}}, B_{j,\vect{p}})$ and  $[\vect{F}]_{i}:=\mathcal{F}\left( B_{i,\vect{p} } \right)$.
 The discrete system matrix $\vect{A} $ can be written as the sum of Kronecker product matrices (see \eqref{eq:equivalent_A})
\begin{equation}
\vect{A} =K_t\otimes M_s + M_t\otimes J_s + W_t \otimes L_s, 
\label{eq:A_kron}
\end{equation}
where the time  matrices are for $i,j=1,\dots, n_t$
 \begin{equation*} [K_t]_{i,j}  :=\int_0^T  b'_{i,p_t}(t)\, b'_{j,p_t}(t)\,\dt,   \qquad
\left[ M_t\right]_{i,j}   :=\int_0^T b_{i,p_t}(t)\,b_{j,p_t}(t)\,\dt,    \qquad
  \left[W_t\right]_{i,j}   :=b_{i,p_t}(T)\, b_{j,p_t}(T), \end{equation*}
 and the spatial matrices are for $i,j=1,\dots, N_s$
  \begin{equation*} \left[J_s\right]_{i,j} :=\int_{\Omega}\Delta B_{i, \vect{p}_s}(\vecx)\, \Delta B_{j, \vect{p}_s}(\vecx)\,\d\Omega, \qquad
  \left[ M_s\right]_{i,j}   :=\int_{\Omega}B_{i,\vect{p}_s}(\vecx)\,B_{j,\vect{p}_s}(\vecx)\,\d \Omega,     \end{equation*}
  \begin{equation*}\left[ L_s\right]_{i,j}  :=\int_{\Omega} \nabla B_{i,\vect{p}_s}(\vecx)\, \nabla B_{j, \vect{p}_s}(\vecx)\,\d\Omega.  \end{equation*}

\subsection{Preconditioner definition and properties }
{ Thanks to the  least-squares formulation of the heat equation,} the matrix $\vect{A}$ in \eqref{eq:A_kron} is symmetric and positive
definite. Thus, we can design and analyze a  suitable   symmetric positive definite
preconditioner  to be used for a preconditioned Conjugate Gradient  method.
 
The simpler version of our preconditioner is associated with the
bilinear form $\widehat{\mathcal{P}}:\widehat{\V}_{h,0} \times
\widehat{\V}_{h,0}\rightarrow \mathbb{R}$ defined as
\begin{equation}
\label{eq:prec_bil}
\widehat{ \mathcal{P}}(w_h,v_h):=\int_{0}^1\int_{\widehat{\Omega}}\partial_{\tau}{w}_h  \,\partial_{\tau}{v}_h \  \d{\widehat{\Omega}}\,\d\tau +  \sum^d_{k=1}\int_{0}^1\int_{\widehat{\Omega}}\frac{\partial^2w_h }{\partial  {\eta}_k^2}   \frac{\partial^2v_h }{\partial  {\eta}_k^2} \ \d\widehat{\Omega} \,\d\tau
\end{equation}
and with the corresponding norm 
\begin{equation}
\label{eq:prec_norm}
\|v_h\|_{\widehat{\mathcal{P}}}^2:= \widehat{\mathcal{P}}(v_h,v_h).
\end{equation}
The preconditioner matrix  is given by
\[[\vect{P}]_{i,j}=\widehat{\mathcal{P}}(\widehat{B}_{i, \vect{p}}(\vect{\eta},\tau), \widehat{B}_{j,\vect{p}}(\vect{\eta},\tau))  \qquad i,j=1,\dots,N_{dof}\]
and has the following structure:
\begin{equation}
\vect{P} = \widehat{K}_t \otimes \widehat{M}_s   + \widehat{M}_t\otimes  \widetilde{J}_s,
\label{eq:preconditioner}
\end{equation}
where, referring to \eqref{eq:basis_par} for the notation of the basis functions, for $ i,j=1,\dots,n_t$
\[[\widehat{K}_t]_{i,j}  :=\int_0^1\widehat{b}'_{i,p_t}(\tau)\ \widehat{b}'_{j,p_t}(\tau)\,\d\tau,   \qquad
 [ \widehat{M}_t ]_{i,j}   :=\int_0^1\widehat{b}_{i,p_t}(\tau)\ \widehat{b}_{j,p_t}(\tau)\,\d\tau ,    \]
 while for $  i,j=1,\dots,N_s$
\[[\widetilde{J}_s]_{i,j} := \sum^d_{k=1} \int_{\widehat{\Omega}}\frac{\partial^2\widehat{B}_{i,\vect{p}_s}(\vect{\eta}) }{\partial  {\eta}_k^2}   \frac{\partial^2\widehat{B}_{j,\vect{p}_s}(\vect{\eta}) }{\partial  {\eta}_k^2} \ \d\widehat{\Omega},  \qquad [\widehat{M}_s]_{i,j} := \int_{\widehat{\Omega}} \widehat{B}_{i,\vect{p}_s}(\vect{\eta})\ \widehat{B}_{j,\vect{p}_s}(\vect{\eta}) \,\d \widehat{\Omega}.\]
Note that $\widehat{K}_t$, $\widehat{M}_t$ and $\widehat{M_s}$
correspond to $ {K}_t$, $ {M}_t$ and $ {M_s}$, respectively,  where
the integration is performed on the parametric domain
$\widehat{\Omega}$. The matrices  $\widetilde{J_s}$ and  $\widehat{M_s}$ can be further factorized as sum of Kronecker products as
\[\widetilde{J}_s = \sum_{k=1}^d \widehat{M}_d\otimes \dots\otimes\widehat{M}_{k+1}\otimes \widehat{J}_k\otimes\widehat{M}_{k-1}\otimes\dots \otimes\widehat{M}_1 , \qquad 
 \widehat{M_s} =\widehat{M}_d \otimes\dots \otimes\widehat{M}_1,\]
where for $k=1,\dots,d$ and for $ i,j=1,\dots,n_{s,k}$
\begin{equation*}
[\widehat{J}_k]_{i,j} := \int_{0}^1 \widehat{b}_{i,p_s}''( {\eta}_k)\ \widehat{b}_{j,p_s}''( {\eta}_k) \,\d{\eta}_k, \qquad [\widehat{M}_k]_{i,j} := \int_{0}^1 \widehat{b}_{i,p_s}( {\eta}_k)\ \widehat{b}_{j,p_s}( {\eta}_k) \,\d {\eta}_k.
\end{equation*}
If $d=3$, that is the case addressed in the numerical tests,  we have that \eqref{eq:preconditioner} becomes
\begin{equation*} \vect{P} = \widehat{K}_t \otimes \widehat{M}_3\otimes \widehat{M}_2 \otimes \widehat{M}_1  + \widehat{M}_t\otimes \widehat{J}_3\otimes \widehat{M}_2 \otimes \widehat{M}_1  + \widehat{M}_t\otimes \widehat{M}_3\otimes \widehat{J}_2 \otimes \widehat{M}_1  + \widehat{M}_t\otimes \widehat{M}_3\otimes \widehat{M}_2 \otimes \widehat{J}_1. \end{equation*}

\subsubsection{Spectral analysis}
\label{sec:spec_est}

We now  focus on the spectral analysis of $\vect{P}^{-1}\vect{A}$.  
We need to define the bilinear form $ \mathcal{P}:\V_{h,0} \times \V_{h,0}\rightarrow \mathbb{R}$ 
\begin{equation*}
\mathcal{P}(w_h,v_h):=\int_{0}^T\int_{\Omega}\partial_t{w}_h\,  \partial_t{v_h}  \ \d\Omega\,\dt +  \sum^d_{k=1}\int_{0}^T\int_{\Omega} \frac{\partial^2w_h}{\partial x_k^2}\frac{\partial^2v_h}{\partial x_k^2}      \ \d\Omega \,\dt
\end{equation*}
and the associated norm
\[
\| v_h \|^2_{\mathcal{P}}:=\mathcal{P}(v_h,v_h).
\]
Note that $\mathcal{P}(\cdot,\cdot)$ and $\| \cdot \|_{\mathcal{P}}$
are analogous to $\mathcal{\widehat{P}}(\cdot, \cdot)$  and
$\|\cdot\|_{\widehat{\mathcal{P}}}$  but  integration is performed
on the   physical domain (see   \eqref{eq:prec_bil} and \eqref{eq:prec_norm}).

 We first prove the equivalence between the norms $\|\cdot\|_{\mathcal{P}}$ and $\|\cdot \|_{\V_{0}}$ in $\V_{h,0}$. 

\begin{prop}
\label{prop:norm_eq}
{  Under Assumptions  \ref{ass:disc_spaces_continuity}--\ref{ass: single-patch-domain}, it holds} 
\[
  \frac{1}{C_{\Delta}}   \|v_h\|_{\mathcal{P}}^2\leq \|v_h\|_{\V_0}^2\leq  d  \|v_h\|_{\mathcal{P}}^2 \quad \forall v_h\in \V_{h,0},
\]
where $C_{\Delta}$ is the constant defined in \eqref{eq:elliptic-reg}.
\end{prop}
\begin{proof} 
  Given $v_h\in\V_{h,0}$, recalling
\eqref{eq:discrete-space-conformity} and thanks to \eqref{eq:elliptic-reg}, we have that  
\begin{align*}
  \sum^d_{k=1}\int_0^T \int_{\Omega}\left| \frac{\partial^2v_h}{\partial x_k^2}\right|^2 \,\d \Omega\,\dt  \,  &\leq  \int_{0}^T\int_{\Omega}\left(   \sum^d_{k,l=1} \left| \frac{\partial^2v_h}{\partial x_k \partial x_l}\right|^2\right) \,\d \Omega \,\dt   =  \int_0^T\left| v_h(\cdot, t) \right|^2_{H^2(\Omega)} \,\dt\\
  & \leq   \int_{0}^T\|v_h(\cdot, t) \|^2_{H^2(\Omega)}\,\dt  \leq
    C_{\Delta} 
    \int_{0}^T\|\Delta v_h(\cdot, t) \|^2_{L^2(\Omega)}\,\dt. 
\end{align*} 
Thus, the first inequality holds.   We  also have 
 \begin{align*}
\int_{0}^T\|\Delta v_h(\cdot, t)  \|^2_{L^2(\Omega)}\,\dt &  =  \sum^d_{k,l=1}   \int_0^T\int_{\Omega} \frac{\partial^2v_h}{\partial  x_k^2}  \frac{\partial^2v_h}{\partial x_l^2} \ \d\Omega\,\dt \\
&\leq \frac{1}{2}\sum^d_{k,l=1} \int_0^T \left[\bigg\|\frac{\partial^2v_h}{\partial  x_k^2}(\cdot, t)  \bigg\|_{L^2(\Omega)}^2  + \bigg\|\frac{\partial^2v_h}{\partial x_l^2}(\cdot, t) \bigg\|_{L^2(\Omega)}^2 \right]\,\dt \nonumber \\
&\leq  d \sum^d_{k=1}\int_{0}^T\bigg\|\frac{\partial^2v_h}{\partial x_k^2}  (\cdot, t) \bigg\|_{L^2(\Omega)}^2\,\dt  =d  \sum^d_{k=1}\int_0^T\int_{\Omega}\left| \frac{\partial^2v_h}{\partial x_k^2}\right|^2\,\d\Omega\,\dt 
\end{align*} 
 and  we can conclude that the second inequality holds. 
\end{proof}

\begin{wn}
{  Under Assumptions  \ref{ass:disc_spaces_continuity}--\ref{ass: single-patch-domain}, it holds}  
\begin{equation}
  \frac{1}{C_{\Delta}}   \|v_h\|_{\mathcal{P}}^2\leq \mathcal{A}(v_h, v_h)\leq 2 d \|v_h\|_{\mathcal{P}}^2 \qquad \forall v_h\in\V_{h,0}.\label{eq:equiv_P}
\end{equation}
\end{wn}
\begin{proof}
The statement follows from Lemma \ref{lemma:cont}, Lemma \ref{lemma:coer} and Proposition \ref{prop:norm_eq}.
\end{proof}

\begin{prop}
\label{prop:norm_eq-reference}
{  Under Assumptions   \ref{ass:disc_spaces_continuity}--\ref{ass:  regularpatch-domain},   there}  exist   constants $Q_1,Q_2>0$ independent of  $h_s$, $h_t$,  $p_s$,
$p_t$,  but dependent on $\mathbf{G}$ 
 such that 
\[Q_1 \|{v}_h \|_{\mathcal{P}}^2\leq
\|\widehat{v}_h \|^2_{\widehat{\mathcal{P}}} \leq Q_2
\|{v}_h\|_{\mathcal{P}}^2 \quad \forall \widehat{v}_h \ \in \widehat{\V}_{h,0} \text{ and } v_h:=\widehat{v}_h\circ\mathbf{G}^{-1}.\]
\end{prop}

\begin{proof}  

Let $\widehat{v}_h\in\widehat{\V}_{h,0}$ and ${v}_h:=\widehat{v}_h\circ\mathbf{G}^{-1}\in{\V}_{h,0}.$
 First we   prove the first inequality.
Observing that  $\mathbf{G}^{-1}(\vect{x},t)=(\mathbf{F}^{-1}(\vect{x}),t/ T )$, we get
 \begin{align*} 
\int_0^T\int_{\Omega}(\partial_t{v}_h)^2 \ \d\Omega\ \dt & =\frac{1}{T}\int_0^1\int_{\widehat{\Omega}}( \partial_{\tau}\widehat{v}_h)^2\left| \det\left( J_{\F} \right) \right|\ \d\widehat{\Omega} \ \d\tau \leq \frac{1}{T} \sup_{\widehat{\Omega}}\left\{   \left| \det(J_{\F})\right|\right\}\int_0^1 \| \partial_{\tau}{\widehat{v}}_h (\cdot, \tau) \|^2 _{L^2(\widehat{\Omega})} \ \d\tau\\
 & \leq \frac{1}{T}\sup_{\widehat{\Omega}}\left\{  \left| \det(J_{\F})\right|\right\} \| \widehat{v}_h\|^2_{\widehat{\mathcal{P}}}.
\end{align*} 

Let $\vect{H}_{\widehat{v}_h}$ be the Hessian of
$\widehat{v}_h$  with respect to the spatial parametric variables $\eta_1, \dots,   \eta_d$, i.e. $\vect{H}_{\widehat{v}_h}\in \mathbb{R}^{d\times d}$ with  $[\vect{H}_{\widehat{v}_h}]_{i,j} = \frac{\partial^2\widehat{v}_h}{\partial \eta_i \partial \eta_j} $ for $i,j=1,\dots,d$, and let  $[J_{\F}^{-1}]_{\cdot,i}\in\mathbb{R}^d$  denote the $i$-$th$ column of $J_{\F}^{-1}$. Then, for $i=1,\dots,d$, it holds

  \begin{align*}
\int_{0}^T\int_{\Omega}  \left( \frac{\partial^2 v_h}{\partial x_i^2}\right)^2   \d\Omega \  \dt & = \int_0^1\int_{\widehat{\Omega}} \left([J_{\F}^{-1}]_{\cdot, i}^T \vect{H}_{\widehat{v}_h} [J_{\F}^{-1}]_{\cdot, i}  +\nabla \widehat{v}_h^T \frac{\partial [J_{\F}^{-1}]_{\cdot, i}}{\partial \eta_i}\right)^2  T| \det(J_{\F})|  \ \d\widehat{\Omega}\ \d\tau \\ 
&  \leq \int_0^1\int_{\widehat{\Omega}} \left( \widehat{C}_1
  \|\vect{H}_{\widehat{v}_h}\|_F ^2  + \widehat{C}_2   \|\nabla
  \widehat{v}_h\|_{ 2}  ^2 \right) \ \d\widehat{\Omega} \ \d\tau, 
\end{align*}

where  $\|\cdot\|_F$ and $\|\cdot\|_2$ denote the Frobenius norm and the  two-norm    of matrices { (the norm induced by the Euclidean vector norm),} respectively,  
$$\widehat{C}_1:=2T \max_{i} \sup_{\widehat{\Omega}}\left\{\left(\|[J_{\F}^{-1}]_{\cdot, i} \|_2\right)^4 | \det(J_{\F})|\right\},\quad\widehat{C}_2:= 2T \max_{i} \sup_{\widehat{\Omega}}\left\{\left(\bigg\| \frac{\partial [J_{\F}^{-1}]_{\cdot, i}}{\partial \eta_i}\bigg\|_2\right)^2 | \det(J_{\F})|\right\}$$
 and where we used that $ \|\vect{H}_{\widehat{v}_h}\|_2 \leq \|\vect{H}_{\widehat{v}_h}\|_F $.  

Following  the proof of Proposition \ref{prop:norm_eq},  we can prove that 
 \[\int_{0}^1\|\Delta \widehat{v}_h(\cdot, \tau) \|_{L^2(\widehat{\Omega})}^2\,\d\tau\leq d\|\widehat{v}_h\|_{\widehat{\mathcal{P}}}^2 \quad \forall \widehat{v}_h\in \widehat{\V}_{h,0}.\]
Thus  it holds 
\begin{align*}
\int_0^1\int_{\widehat{\Omega}}  \|\vect{H}_{\widehat{v}_h}\|_F^2 \ \d\widehat{\Omega}\ \d\tau &\leq  2     \int_0^1  | \widehat{v}_h(\cdot, \tau) |_{H^2(\widehat{\Omega})}^2 \ \d\tau   
   \leq 2\widehat{C}_{\Delta} 
     \int_0^1  \|\Delta \widehat{v}_h(\cdot, \tau)  \|_{L^2(\widehat{\Omega})}^2 \ \d\tau\   \leq  2d\widehat{C}_{\Delta}  
        \| \widehat{v}_h  \|_{\widehat{\mathcal{P}}}^2 \  
        \end{align*}
\begin{align*} \int_0^1\int_{\widehat{\Omega}}  \|\nabla \widehat{v}_h\|_2 ^2  \
  \d\widehat{\Omega}\ \d\tau & \ = 
  \int_0^1|\widehat{v}_h(\cdot, \tau) |_{H^1(\widehat{\Omega})}^2 \
    \d\tau \leq  \widehat{C}_{\Delta}  
      \int_0^1 \|\Delta \widehat{v}_h(\cdot, \tau) \|^2_{L^2(\widehat{\Omega})} \,\d\tau  \leq  d\widehat{C}_{\Delta}   
     \| \widehat{v}_h \|_{\widehat{\mathcal{P}}}^2,
     \end{align*}
where $\widehat{C}_{\Delta}>0$ is the constant such that $\|
z\|_{H^2 (\widehat{\Omega})}^2 \leq \widehat{C}_{\Delta}\|\Delta
z\|_{L^2 (\widehat{\Omega})}^2$, for  $ z \in  H^1_0
(\widehat{\Omega})\cap H^2 (\widehat{\Omega})$.
 Therefore, we  have  
 \begin{align*}
\int_{0}^T\int_{\Omega} \left( \frac{\partial^2 v_h}{\partial x_i^2}\right)^2 \ \d\Omega \  \dt & \leq  d\widehat{C}_{\Delta}
\left( 2\widehat{C}_1 + \widehat{C}_2\right)   \| \widehat{v}_h \|_{\widehat{\mathcal{P}}}^2
\end{align*} 
and, summing  all terms that define $\|\cdot\|_{{\mathcal{P}}}$, we
conclude
\[
Q_1 \|v_h\|_{\mathcal{P}}^2\leq
\|\widehat{v}_h  \|^2_{\widehat{\mathcal{P}}}
\]
with $\frac{1}{Q_1}:=  \frac{1}{T}\sup_{\widehat{\Omega}}\left\{  \left| \det(J_{\F})\right|\right\} + d^2\widehat{C}_{\Delta}\left( 2\widehat{C}_1 + \widehat{C}_2\right)   $.

Now we prove the other bound.  We   observe that $\widehat{v}_h=v_h\circ\mathbf{G}$ and  $\mathbf{G}(\vect{\eta},\tau)=(\mathbf{F}(\vect{\eta}),T\tau )$. Thus,     with  similar arguments and using \eqref{eq:elliptic-reg},   we have
\[  \int_0^1\int_{\widehat{\Omega}}\partial_{\tau}{\widehat{v}}^2_h\ \d \widehat{\Omega}\ \d\tau \leq T \sup_{{\Omega}}\left\{  \left| \det(J_{\F^{-1}})\right|\right\} \| {v}_h\|^2_{{\mathcal{P}}}\]   and  \[
  \int_{0}^1\int_{\widehat{\Omega}} \left( \frac{\partial^2 \widehat{v}_h}{\partial \eta_i^2}\right)^2 \ \d\widehat{\Omega} \  \d\tau  \leq     d  C_{\Delta}\left(  2C_1 + C_2 \right) \| {v}_h \|_{\mathcal{P}}^2,\]
where $ {C}_1:=2\frac{1}{T}\max_{i} {\sup}_{\Omega}\left\{\left(\|[J_{\F^{-1}}^{-1}]_{\cdot, i} \|_2\right)^4 | \det(J_{\F^{-1}})|\right\}$ and \\ $ {C}_2:= 2\frac{1}{T} \max_{i}\sup_{\Omega}\left\{\left(\bigg\| \frac{\partial [J_{\F^{-1}}^{-1}]_{\cdot, i}}{\partial \eta_i}\bigg\|_2\right)^2 | \det(J_{\F^{-1}})|\right\}$.
We conclude that
\[ \|\widehat{v}_h\|_{\widehat{\mathcal{P}}}^2\leq Q_2 \| {v}_h \|_{\mathcal{P}}^2 \]
with ${Q_2}:=  {T}\sup_{ {\Omega}}\left\{  \left| \det(J_{\F^{-1}})\right|\right\} + d^2{C_{\Delta}}  \left(2 {C}_1 +  {C}_2 \right)  $.
\end{proof}
\begin{tw}
\label{teo:spec_est}   
   Under Assumptions  \ref{ass:disc_spaces_continuity}--\ref{ass:  regularpatch-domain},  it   holds
\[ \theta \leq \lambda_{\min}(\vect{P}^{-1}\vect{A}), \qquad \lambda_{\max}(\vect{P}^{-1}\vect{A})\leq \Theta, \]
where $\theta$ and $\Theta$ are positive constants that do not depend on $h_s$, $h_t$, $p_s$ and $p_t$.
\end{tw}
\begin{proof}
Let $\widehat{v}_h\in \widehat{\V}_{h,0}$, ${\mathbf{v}}$ its coordinate vector with respect to the basis \eqref{eq:all_basis} and $v_h=\widehat{v}_h\circ \mathbf{G}^{-1}\in \V_{h,0}$. Thanks to Courant-Fischer theorem, { we have to show that there are bounds $\theta$ and $\Theta$ such that 
\[ \theta \leq \frac{\mathbf{v}^T \vect{A} \mathbf{v}}{\mathbf{v}^T \vect{P} \mathbf{v}}\leq \Theta \]
holds for all $\textbf{v}.$}
Equivalently, using \eqref{eq:equiv_P}  and noting that $\mathbf{v}^T \vect{A} \mathbf{v} = \mathcal{A}(v_h, v_h)$ and  $\mathbf{v}^T \vect{P} \mathbf{v} = \widehat{\mathcal{P}}(\widehat{v}_h, \widehat{v}_h) = \|\widehat{v}_h\|^2_{\widehat{\mathcal{P}}}$,
it is sufficient to {  show that there are bounds   $\theta$ and $\Theta$ such that }
\[  {\theta}{C_{\Delta}}  \leq \frac{\|v_h\|^2_{\mathcal{P}}}{\|\widehat{v}_h\|^2_{\widehat{\mathcal{P}}}} \leq \frac{\Theta}{2 d } \quad \forall \widehat{v}_h\in\widehat{\V}_{h,0}, \]
with $v_h=\widehat{v}_h \circ \mathbf{G}^{-1}$.
Using  Proposition \ref{prop:norm_eq-reference}, we can conclude that the previous inequalities hold with  $\theta:=\frac{1}{C_{\Delta}Q_2}$   and $\Theta:=\frac{2d}{Q_1}$.
\end{proof}

\subsection{Preconditioner implementation by fast diagonalization}

The application of the preconditioner is a solution of a Sylvester-like equation: given $\mathbf{r}$ find $\mathbf{s}$ such that
\begin{equation}
\label{eq:prec_sys}
\vect{P}\mathbf{s}=\mathbf{r}.
\end{equation}

Following \cite{Sangalli2016}, to solve \eqref{eq:prec_sys}, we use  the fast diagonalization (FD) method (see  \cite{Deville2002} and \cite{Lynch1964} for further details). It is a direct method that, at the first step, computes the eigendecomposition of the pencils $(\widehat{M}_i, \widehat{J}_i)$ for $i=1,\dots,d$ and of $(\widehat{M}_t, \widehat{K}_t)$, i.e.  
\begin{equation}
\label{eq:eig_dec}
 \widehat{J}_iU_i=\widehat{M}_iU_i\Lambda_i, \qquad \qquad \widehat{K}_tU_t=\widehat{M}_tU_t\Lambda_t  
\end{equation}
 where $\Lambda_i$ and $\Lambda_t$ are diagonal eigenvalue matrices  while the columns of $U_i$ and $U_t$ contain the corresponding generalized eigenvectors and they are such that
 \[ \widehat{M}_i=U_i^{-T}U_i^{-1}, \qquad \widehat{J}_i=U_i^{-T}\Lambda_i U_i^{-1}, \qquad \widehat{M}_t=U_t^{-T}U_t^{-1}, \qquad \widehat{K}_t=U_t^{-T}\Lambda_t U_t^{-1}.\]
Then, we can rewrite $\widehat{M}_s$   as
\[\begin{array}{rlr}
\widehat{M}_s & = (U_{d}^{-T}U_{d}^{-1})\otimes\dots\otimes(U_1^{-T}U_1^{-1})  = (U_d^{-T}\otimes\dots\otimes U_1^{-T})(U_d^{-1}\otimes\dots\otimes U_1^{-1}), \\ 
& = (U_d \otimes\dots\otimes U_1)^{-T}(U_d \otimes\dots\otimes U_1)^{-1} = U_s^{-T}U_s^{-1},
\end{array}\]
where $U_s := U_d \otimes \dots \otimes U_1$ and where we used \eqref{eq:kron_prod} for the first equality and  \eqref{eq:kron_transp} and \eqref{eq:kron_inv} for the second equality above. Similarly, denoting with $I_m\in\mathbb{R}^{m\times m}$   the identity matrix of size $m$ and defining  $\Lambda_s:=\sum_{i=1}^dI_{n_s^{i-1}}\otimes \Lambda_i \otimes I_{n_s^{d-i}}$, we rewrite $\widetilde{J}_s$ as
 \begin{align*} 
\widetilde{J}_s \! & =  \sum_{i=1}^d (U_d^{-T}U_d^{-1})\otimes\dots\otimes(U_{i+1}^{-T}U_{i+1}^{-1})\otimes(U_{i }^{-T}\Lambda_i U_{i}^{-1})\otimes(U_{i-1}^{-T}U_{i-1}^{-1})  \otimes\cdots\otimes (U_{1}^{-T}U_{1}^{-1})  \\ 
& = \sum_{i=1}^d (U_d^{-T}\otimes\dots\otimes U_{ 1}^{-T}) ( I_{n_s^{i-1}}\otimes \Lambda_i \otimes I_{n_s^{d-i}}) (U_d^{-1}\otimes\dots\otimes U_{ 1}^{-1}) ,  \\
& = U_s^{-T} \otimes\Lambda_s \otimes U_s^{-1},     
\end{align*} 
 where we used      \eqref{eq:kron_prod} for the second equality and \eqref{eq:kron_transp}, \eqref{eq:kron_prod} and \eqref{eq:kron_inv}  for the third equality above.
 Then, $\vect{P}$ can be factorized as  
\[
 \vect{P}=(U_t\otimes U_s)^{-T}( \Lambda_t \otimes   I_{n^d_s}  + I_{n_t}\otimes \Lambda_s)(U_t\otimes U_s)^{-1},
\]
 where we  used  \eqref{eq:kron_transp}, \eqref{eq:kron_prod} and   \eqref{eq:kron_inv}.  
 Therefore,  after introducing the tensors $\mathfrak{R},  \widetilde{\mathfrak{Q}}\in\mathbb{R}^{n_{s,1}\times\dots n_{s,d}\times n_t}$  s.t.  $\text{vec}\left( {\mathfrak{R}}\right)=\mathbf{r}$ and $\text{vec}\left(\widetilde{\mathfrak{Q}}\right)=\mathbf{\widetilde{q}}$, the solution of \eqref{eq:prec_sys} can be obtained by the following algorithm.

\begin{algorithm}[H]
\caption{  FD method }\label{al:direct_P}
\begin{algorithmic}[1]
\State Compute the generalized eigendecompositions \eqref{eq:eig_dec}.
\State Compute $\widetilde{\mathbf{r}} = (U_t\otimes U_s)^T \mathbf{r}  = {(U_t\otimes U_d \otimes \dots\otimes U_1)^T \mathbf{r}}  = \mathfrak{R}\times_1U_1^T\dots\times_{d+1}U_t^T $.
\State Compute $\widetilde{\mathbf{q}} = \left(\Lambda_t \otimes  I_{n^d_s} + I_{n_t}\otimes \Lambda_s \right)^{-1} \widetilde{\mathbf{r}}. $
\State Compute $\mathbf{s} =  (U_t\otimes U_s) \ \widetilde{\mathbf{q}} =  (U_t\otimes U_d \otimes \dots\otimes U_1)\ \widetilde{\mathbf{q}}   = \widetilde{\mathfrak{Q}}\times_1 U_1\dots\times_{d+1}U_t . $
\end{algorithmic}
\end{algorithm}

\subsection{Inclusion of the geometry information in the preconditioner}
\label{sec:geo_inclusion}
The spectral estimates in Section \ref{sec:spec_est} show the
dependence on   $\mathbf{G}$ (see the proof of  Theorem \ref{teo:spec_est}):
 the geometry parametrization affects the performance of our
 preconditioner \eqref{eq:preconditioner}, as it is confirmed  by the numerical tests in Section \ref{sec:num_test}. 
In this section, we present a strategy to partially incorporate
$\mathbf{G}$ in the preconditioner, without increasing its
computational cost. The same idea has been used in
\cite{Montardini2018} for the Stokes problem.

We begin by splitting the bilinear form $\mathcal{A}(\cdot, \cdot)$   as 
\begin{equation*}
 \mathcal{A}(v_h, w_h) =   \mathcal{K}_t (v_h,w_h) + \mathcal{K}_s(v_h,w_h)  - \mathcal{O}(v_h,w_h) \qquad \forall v_h,w_h\in \V_{h,0}
\end{equation*}
 where 
 \begin{equation*}
  \mathcal{K}_t(v_h,w_h)  \ := \  \int_0^T\int_{\Omega} \partial_{t}v_h\,\partial_{t}w_h\ \d\Omega \ \dt,   \qquad 
  \mathcal{K}_s(v_h,w_h)  \ := \  \int_0^T \int_{\Omega} \Delta  v_h\, \Delta  w_h \ \d\Omega \ \dt,   
\end{equation*}
 \begin{equation*}
  \mathcal{O}(v_h,w_h)  :=  \int_0^T \int_{\Omega}( \partial_{t}{v}_h\,\Delta w_h + \partial_t{w}_h\,\Delta v_h) \ \d\Omega \ \dt.  
 \end{equation*}
Using that ${v}_h:=\widehat{v}_h\circ\mathbf{G}^{-1}$, ${w}_h:=\widehat{w}_h\circ\mathbf{G}^{-1}$ and 
 \begin{equation*}
  \frac{\partial^2 v_h}{\partial x_i^2}= \sum_{j,k=1}^d \frac{\partial^2 \widehat{v}_h \circ\mathbf{G}^{-1} }{\partial \eta_j  \partial \eta_k}   [J_{\F }^{-1}]_{ki}[J_{\F }^{-1}]_{ji}+ \sum_{j=1}^d\frac{\partial \widehat{v}_h \circ \mathbf{G}^{-1}}{\partial \eta_{j}}\frac{\partial [J_{\F }^{-1}]_{ji}}{\partial \eta_i },
   \end{equation*} we can rewrite $\mathcal{K}_t$ and $\mathcal{K}_s$ as
 \begin{equation*}
\mathcal{K}_t(v_h,w_h)  \  =   \int_{0}^1\int_{\widehat{\Omega}}    c_{d+1}  \partial_{\tau}\widehat{v}_h\, \partial_{\tau}\widehat{w}_h  \ \d\widehat{\Omega}\ \d\tau \quad
\mathcal{K}_s(v_h,w_h) \ =  \mathcal{K}_{s,1}(\widehat{v}_h,\widehat{w}_h)\   + \mathcal{K}_{s,2} (\widehat{v}_h,\widehat{w}_h) ,
 \end{equation*}
where 
\begin{align*}
  \mathcal{K}_{s,1}(\widehat{v}_h,\widehat{w}_h)  := & \sum_{k=1}^d \int_0^1\int_{\widehat{\Omega}} c_k\frac{\partial^2 \widehat{v}_h }{\partial \eta_k^2} \frac{\partial^2 \widehat{w}_h }{\partial \eta_k^2} \ \d\widehat{\Omega} \ \d\tau, \\
   \mathcal{K}_{s,2}(\widehat{v}_h,\widehat{w}_h)  := & \sum_{\substack{r,s=1 \\ r\neq s}}^d \sum_{\substack{j,k=1 \\ j\neq k}}^d \int_0^1\int_{\widehat{\Omega}}   g^1_{rsjk} \frac{\partial^2 \widehat{v}_h }{\partial \eta_k \partial \eta_j} \frac{\partial^2 \widehat{w}_h }{\partial \eta_r \partial \eta_s}   \, \d\widehat{\Omega}\,\dt  + \sum_{\substack{j,k=1  }}^d\int_{0}^1\int_{\widehat{\Omega}}  g^2_{jk} \frac{\partial  \widehat{v}_h }{\partial \eta_k  } \frac{\partial  \widehat{w}_h }{\partial \eta_j }    \,\d\widehat{\Omega}\,\dt \\
   & + \sum_{ {r =1  }}^d \sum_{\substack{j,k=1  }}^d\int_{0}^1\int_{\widehat{\Omega}}  g^3_{rjk} \left( \frac{\partial^2 \widehat{v}_h }{\partial \eta_k \partial \eta_j} \frac{\partial  \widehat{w}_h }{\partial \eta_r  } +\frac{\partial^2 \widehat{w}_h }{\partial \eta_k \partial \eta_j} \frac{\partial  \widehat{v}_h }{\partial \eta_r  }   \right)  \,\d\widehat{\Omega}\,\dt 
\end{align*} 
and where we have defined
\begin{equation}
  \label{eq:coefficients-c}
  c_k:= \left(\big\|   [J_{\F}^{-1} ]_{\cdot,k} \big\|_2\right)^4  |\det(J_{\mathbf{F}})|T \quad\text{for $k=1,\dots,d$}, \quad  c_{d+1}:=|\det(J_{\mathbf{F}})|T^{-1}, 
\end{equation}
while   $g^1_{rsjk}, g^2_{jk},  g^3_{rjk}$ are functions that depend on the parametrization $\mathbf{G}$.

The preconditioner will be based on an approximation of
$\mathcal{K}_t + \mathcal{K}_{s,1}$ only. 
In particular we  approximate $c_{k}$, for 
 $k=1,\dots,d+1$ as 
\begin{equation}
  \label{eq:coefficient-approx}
  \begin{aligned}
  &  c_{k}(\vect{\eta},\tau)  \approx
  \mu_1(\eta_1)\dots\mu_{k-1}(\eta_{k-1})\omega_k(\eta_k)\mu_{k+1}(\eta_{k+1})\dots\mu_d(\eta_d)  \mu_{d+1}(\tau)   \quad k=1,\dots,d, 
  \\  & c_{d+1}(\vect{\eta},\tau) \approx \mu_1(\eta_1)\dots\mu_d(\eta_d) \omega_{d+1}(\tau). 
  \end{aligned}
\end{equation}
The functions $c_{k}$ 
 in \eqref{eq:coefficient-approx}
are first interpolated  by constants in each element and then the
construction of the univariate factors $ \mu_{k}$, and $ \omega_{k}$
is performed by the separation of variable algorithm detailed in  the Appendix \ref{app:sep_var}.
 The  resulting computational cost is therefore proportional to the number of elements, 
  which for smooth  splines is roughly equal to $N_{dof}$, and independent of the degrees $p_s$ and $p_t$.
 
As a consequence,  
 the computation of  \eqref{eq:coefficient-approx} has a negligible cost in the whole iterative strategy.
This first  step leads to a matrix of this form
 \begin{equation*}
\overline{\vect{P}^{\mathbf{G}}} :=  \widehat{K}^{\mathbf{G}}_t \otimes  \widehat{M}^{\mathbf{G}}_s + \widehat{M}^{\mathbf{G}}_t\otimes \widetilde{J}_s^{\mathbf{G}}, 
 \end{equation*}
where, referring to \eqref{eq:basis_par} for the notation of the basis functions,
 \begin{equation*}
\left[ \widehat{K}^{\mathbf{G}}_t\right]_{i,j}  :=\int_0^1 \omega_{d+1} (\tau)\,\widehat{b}'_{i,p_t}(\tau)\,\widehat{b}'_{j,p_t}(\tau)\,\d\tau,   
\quad 
\left[\widehat{M}^{\mathbf{G}}_t\right]_{i,j}   :=\int_0^1 \mu_{d+1} (\tau)\,\widehat{b}_{i,p_t}(\tau)\,\widehat{b}_{j,p_t}(\tau)\,\d\tau \quad    i,j=1,\dots,n_t,  
\end{equation*}
 \begin{equation*}
\widetilde{J}^{\mathbf{G}}_s := \sum_{k=1}^d \widehat{M}^{\mathbf{G}}_d\otimes \dots\otimes\widehat{M}^{\mathbf{G}}_{k+1}\otimes \widehat{J}^{\mathbf{G}}_k\otimes\widehat{M}^{\mathbf{G}}_{k-1}\otimes\dots \otimes\widehat{M}^{\mathbf{G}}_1 , \quad 
 \widehat{M}_s^{\mathbf{G}} :=\widehat{M}^{\mathbf{G}}_d \otimes\dots \otimes\widehat{M}^{\mathbf{G}}_1,
 \end{equation*}
with for $i,j=1,\dots,n_{s,k} $ and $k=1,\dots,d$,
 \begin{equation*}
[ \widehat{J}^{\mathbf{G}}_k]_{i,j} := \int_{0}^1 \omega_k(\eta_k)\, \widehat{b}_{i,p_s}''( {\eta}_k)\,\widehat{b}_{j,p_s}''( {\eta}_k) \,\d{\eta}_k, \quad  [ \widehat{M}^{\mathbf{G}}_k]_{i,j} := \int_{0}^1 \mu_k(\eta_k)\, \widehat{b}_{i,p_s}( {\eta}_k)\,\widehat{b}_{j,p_s}( {\eta}_k) \,\d {\eta}_k.
 \end{equation*}
The matrix $\overline{\vect{P}^{\mathbf{G}}}$ maintains the Kronecker structure of \eqref{eq:preconditioner} and     Algorithm \ref{al:direct_P} can still be used to compute its application.

Finally, as in \cite{Montardini2018}, we apply a diagonal scaling and  we define the   preconditioner as $\vect{P}^{\mathbf{G}}:=\vect{D}^{1/2}\ \overline{\vect{P}^{\mathbf{G}}}\ \vect{D}^{1/2}$ where $\vect{D}$ is the diagonal matrix whose diagonal entries are $[\vect{D}]_{i,i}:=[\vect{A}]_{i,i}/[\overline{\vect{P}^{\mathbf{G}}}]_{i,i}$.

\begin{rmk}
For the model problem considered in this paper,  the approximation of the geometry parametrization  in the time direction is
trivial. Notice that the coefficients in  \eqref{eq:coefficients-c} do
not depend on $\tau$.   Indeed, in our case it holds  
\[ K_t =  \frac{1}{T}    \widehat{K}_t, \qquad M_t = T \widehat{M}_t, \]\B
and hence we could set  explicitly $\widehat{K}_t^{\mathbf{G}} = K_t$
and $\widehat{M}_t^{\mathbf{G}} = M_t $, which is exact.   However, 
we want to present the more general approximating strategy above which could be used
also when the spatial geometry or  equation's coefficients depend on time.
\end{rmk}

\subsection{Computational cost and memory consumption of the linear solver}
 
The cost of our preconditioning strategies consists of two parts: setup cost and application cost.
 
The setup cost of both $\vect{P}$ and $\vect{P}^{\mathbf{G}}$ includes the eigendecomposition
of the pencils $(\widehat{J}_i, \widehat{M}_i)$ and  $(\widehat{K}_t, \widehat{M}_t)$ or $(\widehat{J}^{\mathbf{G}}_i, \widehat{M}^{\mathbf{G}}_i)$   and $(\widehat{K}^{\mathbf{G}}_t, \widehat{M}^{\mathbf{G}}_t)$, respectively, that is,  Step 1 of Algorithm
1. If we assume for simplicity that $\widehat{J}_i, \widehat{M}_i$, $\widehat{J}^{\mathbf{G}}_i, \widehat{M}^{\mathbf{G}}_i$ for $i=1,\dots,d$ have size $n_s 
\times n_s$ and that  $\widehat{K}_t$, $\widehat{M}_t$, $\widehat{K}_t^{\mathbf{G}}$ and $\widehat{M}_t^{\mathbf{G}}$  have size $n_t \times n_t$,
then the cost of the eigendecomposition is   $O(dn_s^3+n_t^3)$ FLOPs. {    This cost is optimal  for  $d=2$
and negligible for $d=3$, provided that $n_t\approx n_s$}. For $\vect{P}^{\mathbf{G}}$, we also have to include in the setup cost the creation of the diagonal matrix $\vect{D}$, which is negligible,    and 
the construction  of
the $2(d+1)$ univariate approximations $\mu_1, \ldots,   \mu_{d+1}$  
and $\omega_1, \ldots,  \omega_{d+1} $, that are used to
incorporate some geometry information into the preconditioner. As
explained in   Section \ref{sec:geo_inclusion}, this has a cost which
is $O(N_{dof})$ FLOPs.

The application of   $\vect{P}$ and   $\overline{\vect{P}^{\mathbf{G}}}$,    is performed by
Algorithm \ref{al:direct_P},  Steps 2--4.   First we note that the time matrices $\widehat{K}_t, \widehat{M}_t, \widehat{K}^{\mathbf{G}}_t, \widehat{M}^{\mathbf{G}}_t $ and  the spatial matrices $\widehat{J}_i, \widehat{M}_i, \widehat{J}^{\mathbf{G}}_i,$ $ \widehat{M}^{\mathbf{G}}_i$ for $i=1,\dots,d $ are banded matrices with band of width $2p_t+1$ and $2p_s+1$, respectively.  
Then, Step 2  and Step 4 are efficiently performed exploiting property \eqref{eq:kron_vec_multi} and they need a total of $4(dn_s^{d+1}n_t+n_t^2n_s^{d})=4N_{dof}(dn_s+n_t)$ FLOPs, while  Step 3  has  an  optimal  cost, as it requires $O(N_{dof})$ FLOPs.
 Thus, the total cost of Algorithm 1 is $4N_{dof}(dn_s+n_t)+O(N_{dof}) $ FLOPs. 
 The non-optimal dominant cost is given by the dense matrix-matrix
 products of Step 2 and Step 4, which, however, are usually
 implemented on modern computers in a high-efficient way, as they are
 BLAS level 3 operations. 
In our numerical tests, the overall serial computational 
 time grows almost as $O(N_{dof})$ up to the largest problem considered, 
as we will show in  Section \ref{sec:num_test}.

Clearly, the computational cost of each iteration of the CG solver
depends on both the preconditioner application and the residual computation.
  For the sake of completeness, we also discuss the 
cost of the residual computation, which consists in the multiplication between  $\vect{A}$
and a vector.  
Note that this multiplication can be   computed by exploiting the
special structure \eqref{eq:A_kron} and the formula
\eqref{eq:kron_vec}. In this case, we do not need to compute and store the whole matrix $\vect{A}$, but only its factors $K_t$, $W_t$, $M_t$, $J_s$, $L_s$ and $M_s$.     With this matrix-free approach, noting  that the time matrices $K_t$, $W_t$, $M_t$ are banded matrices with a band of width $2p_t + 1$ and the spatial  matrices $J_s$, $L_s$, $M_s$  have a number of non-zeros per row approximately equal to $(2p_s + 1)^d$,  the computational cost of a single matrix-vector product is $6\left[(2p_s+1)^{d}+2p_t+1 \right]N_{dof}\approx 6(2p+1)^{d} N_{dof}$, if   $p = p_s \approx  p_t$.   
Even if this cost is lower than what one would get by using $\vect{A}$
explicitly, the comparison with the cost of the preconditioner  shows that the residual computation easily turns out to be the dominant cost of the iterative solver (see Table \ref{tab:appl_matrix} in Section \ref{sec:num_test}). This issue was already recognized in \cite{Sangalli2016,Montardini2018}.

We now analyze the memory consumption. For the preconditioner, we need to store the eigenvector matrices $U_t, U_1, \ldots, U_d$ and the diagonal eigenvalue matrix $\left(\Lambda_t \otimes  I_{n^d_s} + I_{n_t}\otimes \Lambda_s \right)$. The memory required is 
\[ n_t^2 + d n_s^2 + N_{dof}. \]
For the system matrix, we need to store the matrices $K_t$, $M_t$, $M_s$, $L_s$ and $J_s$ (the storage of $W_t$ is negligible). The memory required is roughly
\[ 2 \left(2 p_t +1\right) n_t + 3 \left(2 p_s +1\right)^d N_s. \]
These numbers show that memory-wise our space-time strategy is very
appealing when compared to other approaches, even when space and time
variables are discretized separately, e.g.,   with finite differences in
time or other  time-stepping schemes. 
To see this, take  $d=3$ and $p_t \approx p_s = p$, and assume  $n^2_t
\leq C p^3 N_s$.  In this case, the total memory consumption is then
$ O\left( p^3 N_s + N_{dof} \right)$ which is the memory required to store the Galerkin matrices associated to   spatial variables, \B  plus the memory required to store the solution of the problem.

We emphasize that it is possible, though beyond the scope of this paper, to take the matrix-free paradigm one step further by using the approach developed in \cite{Sangalli2017}. Using this approach, where even the factors of $\vect{A}$ as in \eqref{eq:A_kron} are not needed, would significantly improve the overall iterative solver in terms of memory and computational cost (both for the setup and for the matrix-vector computations).

\section{Numerical benchmarks} 
\label{sec:num_test}
In this Section, we show numerical experiments that confirm the
convergence behaviour \eqref{eq:a_priori_estimate} of the
least-squares approximation method defined in Section
\ref{sec:least_squares}, and then
  we present some numerical results regarding the performance of our preconditioner.
   
The tests are performed with   Matlab R2015a \B  and  GeoPDEs toolbox \cite{Vazquez2016}, on a Intel Core
 i7-5820K processor, running at 3.30 GHz, with 64 GB of RAM.

In Algorithm \ref{al:direct_P}, the eigendecomposition of Step 1 is done by \texttt{eig} Matlab function, while the multiplications of Kronecker matrices, appearing in Step 2 and 4, are performed by Tensorlab toolbox \cite{Sorber2014}.
 We fix the tolerance of CG equal to  $10^{-8}$ and the initial guess equal to the null vector in all tests. 
 
We set  $h_s=h_t=:h$, and we denote the number of
 { subdivisions in each parametric direction  by $n_{sub}$.}

\begin{figure}
 \centering
 \subfloat[][Quarter of annulus.\label{fig:quarter}]
   {\includegraphics[scale=0.17]{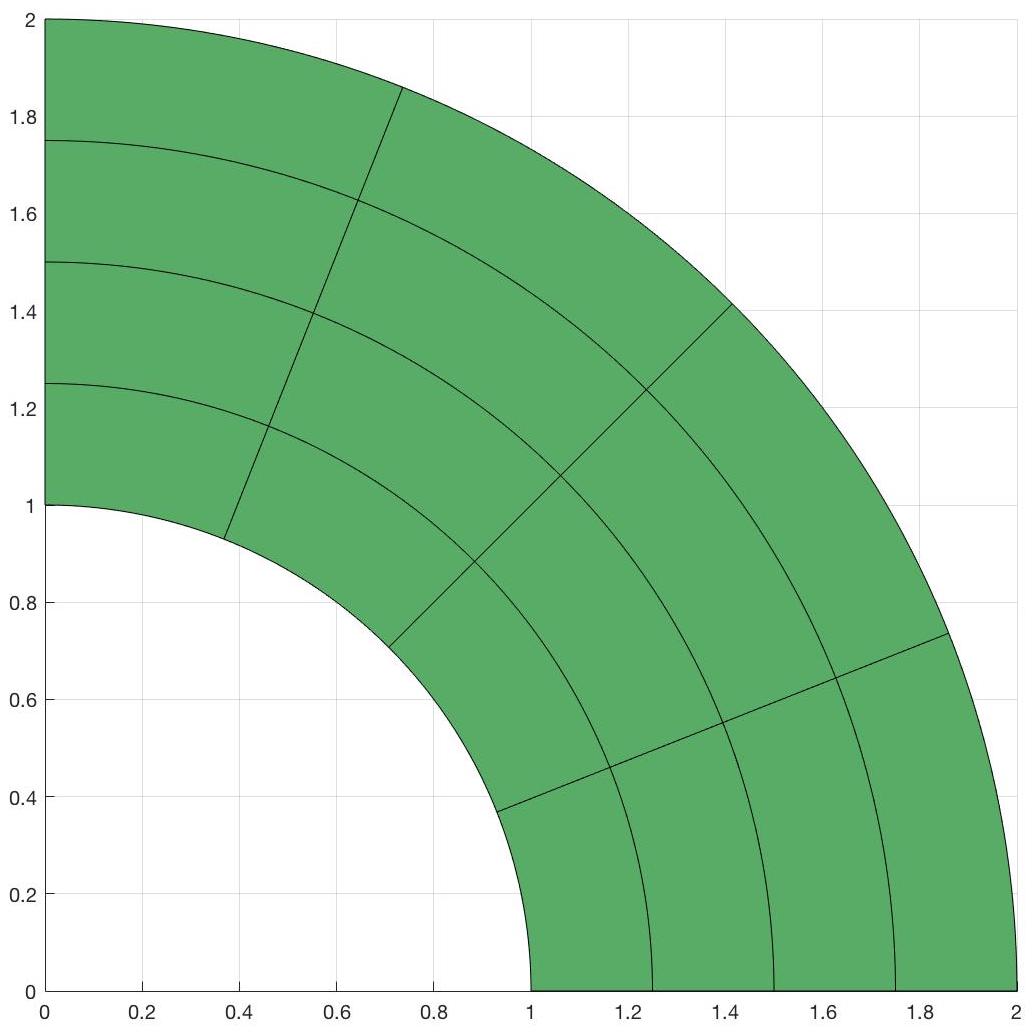}}\quad
 \subfloat[][Cube.\label{fig:cube}]
   {\includegraphics[ scale=0.17]{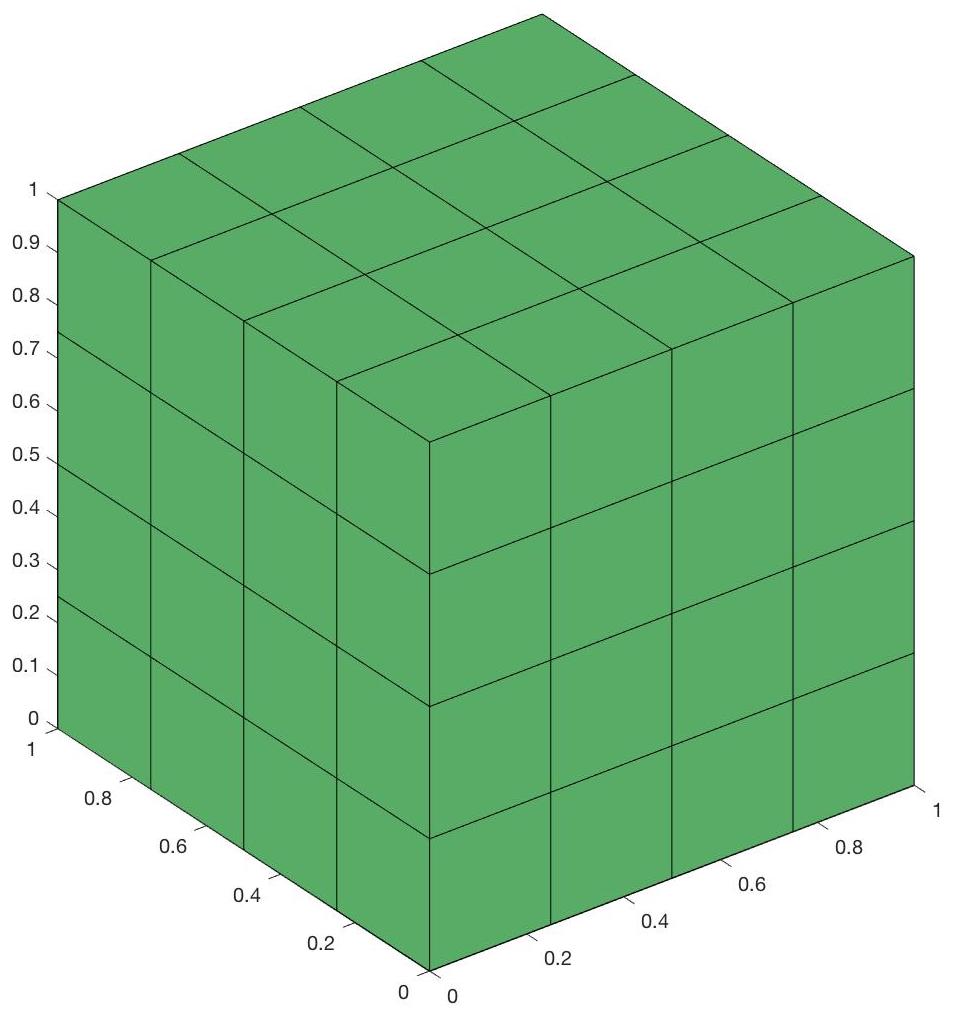}}\\
 \subfloat[][Rotated quarter of annulus.\label{fig:rev_quarter}]
   {\includegraphics[ scale=0.20]{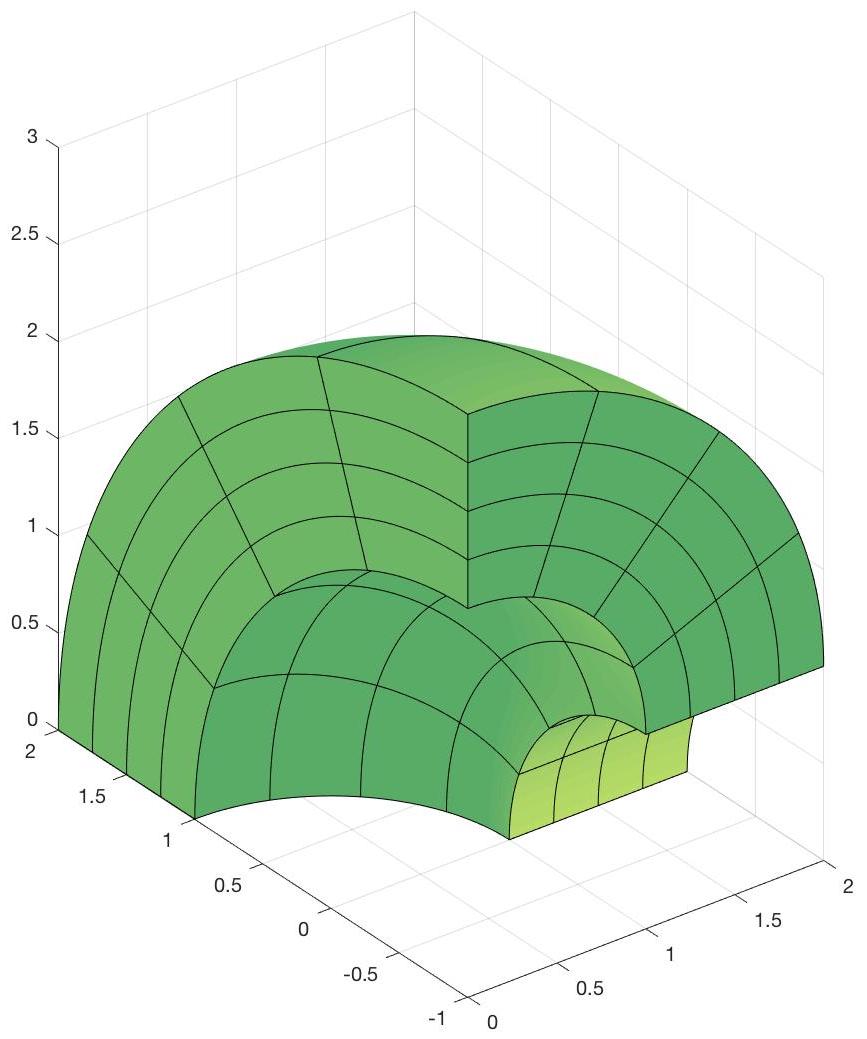}}
 \caption{Computational domains.  }
 \label{fig:Geometries}
\end{figure}
\subsection{Orders of convergence.}  We perform accuracy tests in a
2D spatial domain since the calculation of the numerical errors on 3D
spatial domains is expensive in terms of computational time, when
element-wise Gaussian quadrature is adopted.  
We set $T=1$ and we consider a 2D spatial domain: the \emph{quarter of annulus} with internal radius equal to 1 and external radius equal to 2 (see Figure \ref{fig:quarter}).
The initial and Dirichlet boundary conditions  and the source term $f$  are fixed such that the exact solution is $u =  -(x^2+y^2-1)(x^2+y^2-4)xy^2\sin(\pi t)$.   We solved the linear system with   Matlab direct solver (backslash `` \textbackslash '' operator). 
 Figure \ref{fig:space-time_error_V0} shows the $\|\cdot\|_{\mathcal{V}_0}$ relative errors with splines of degree $p_s=p_t$ from 2 to 6: the rate  of convergence of $O(h^{p_t-1})$ confirms the results of Theorem \ref{teo:a-priori}. 
As predicted by the theory, if we increase  the degree of spatial B-splines   and we set $p_s = p_t+1$, we can gain an order of convergence. Indeed,  Figure \ref{fig:space-time_error_V0_unbalanced}  shows that in this case the $\|\cdot\|_{\mathcal{V}_0}$ relative errors have order $p_t$.

Even if theoretical results do not cover this case, we also analyze in  Figures \ref{fig:space-time_error_L2} and \ref{fig:space-time_error_H1} the error behaviour for $p_t=p_s$ in    $L^2(\Omega\times [0,T])$ and $H^1(\Omega\times [0,T])$ norms, respectively. 
 While the $H^1$ errors are optimal for every $p_t$ considered, i.e. they are of order ${p_t}$ for $p_t\geq 2$, the orders of convergence in $L^2$ norm  are optimal and thus equal to $p_t+1$, only for $p_t\geq 3$.
 The suboptimal behaviour of the error in  $L^2$ norm for $p_t=p_s=2$
is in fact consistent with the Aubin-Nitsche type estimate
and with the a-priori error estimates for fourth-order PDEs  (see in
particular the classical  result  \cite[Theorem 3.7]{Strang1973}).

\begin{figure}
 \centering
 \subfloat[][$\mathcal{V}_0$-norm relative error with $p_t=p_s$.\label{fig:space-time_error_V0}]
   {\includegraphics[width=.48\textwidth]{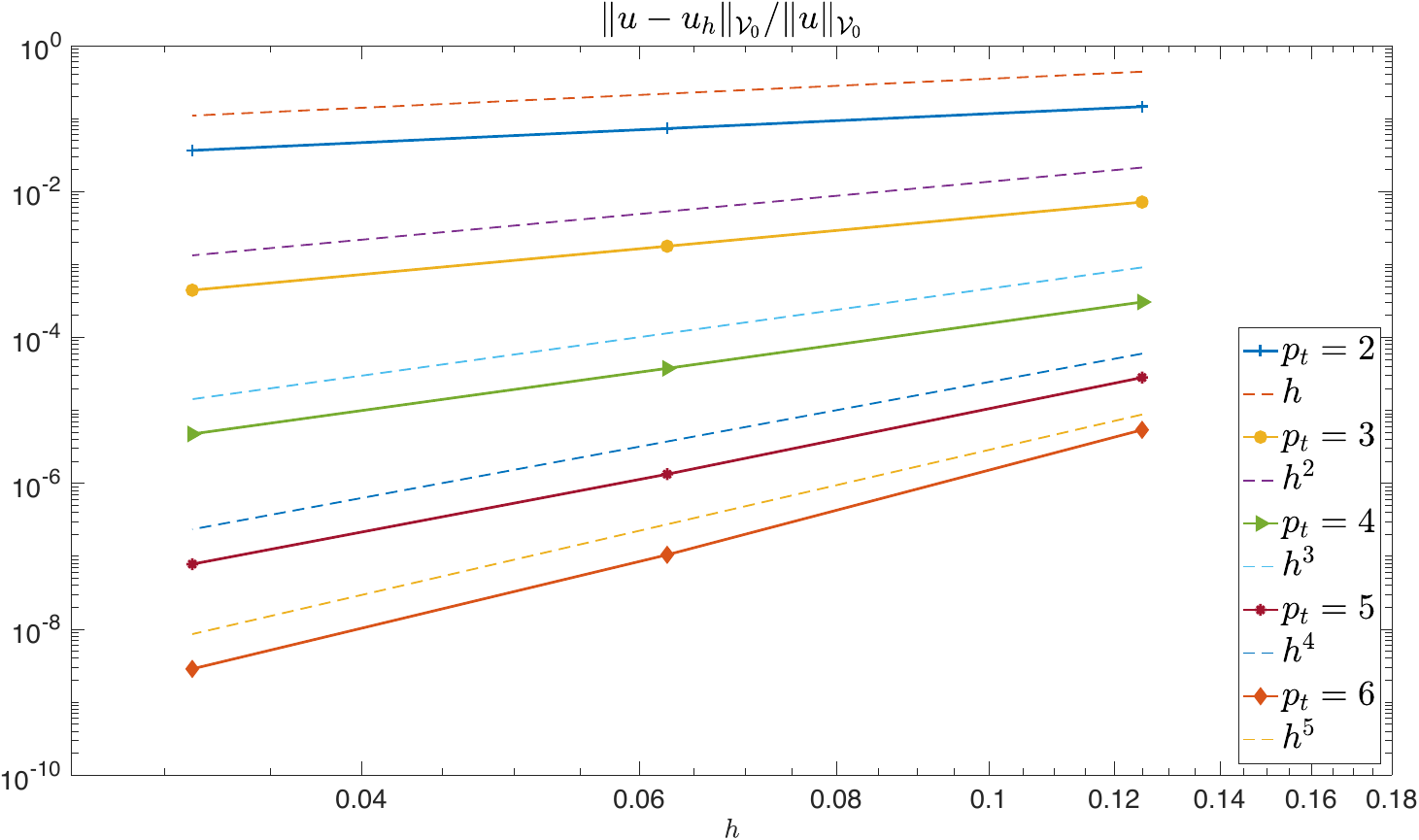}}\quad   
 \subfloat[][$\mathcal{V}_0$-norm relative error with $p_t=p_s-1$.\label{fig:space-time_error_V0_unbalanced}]
   {\includegraphics[width=.48\textwidth]{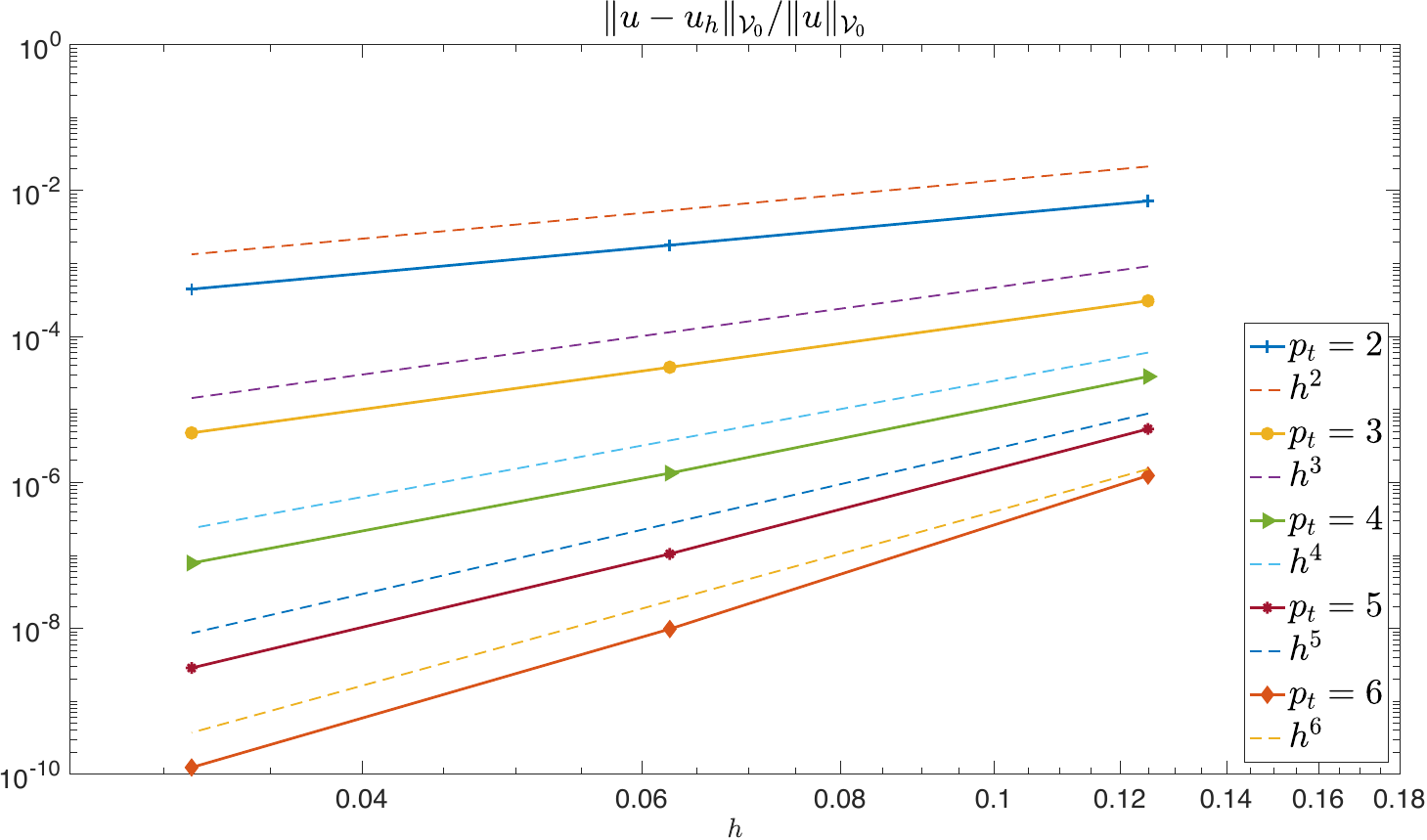}}\\
  \subfloat[][$L^2$-norm relative error with $p_t=p_s$.\label{fig:space-time_error_L2}]
   {\includegraphics[width=.48\textwidth]{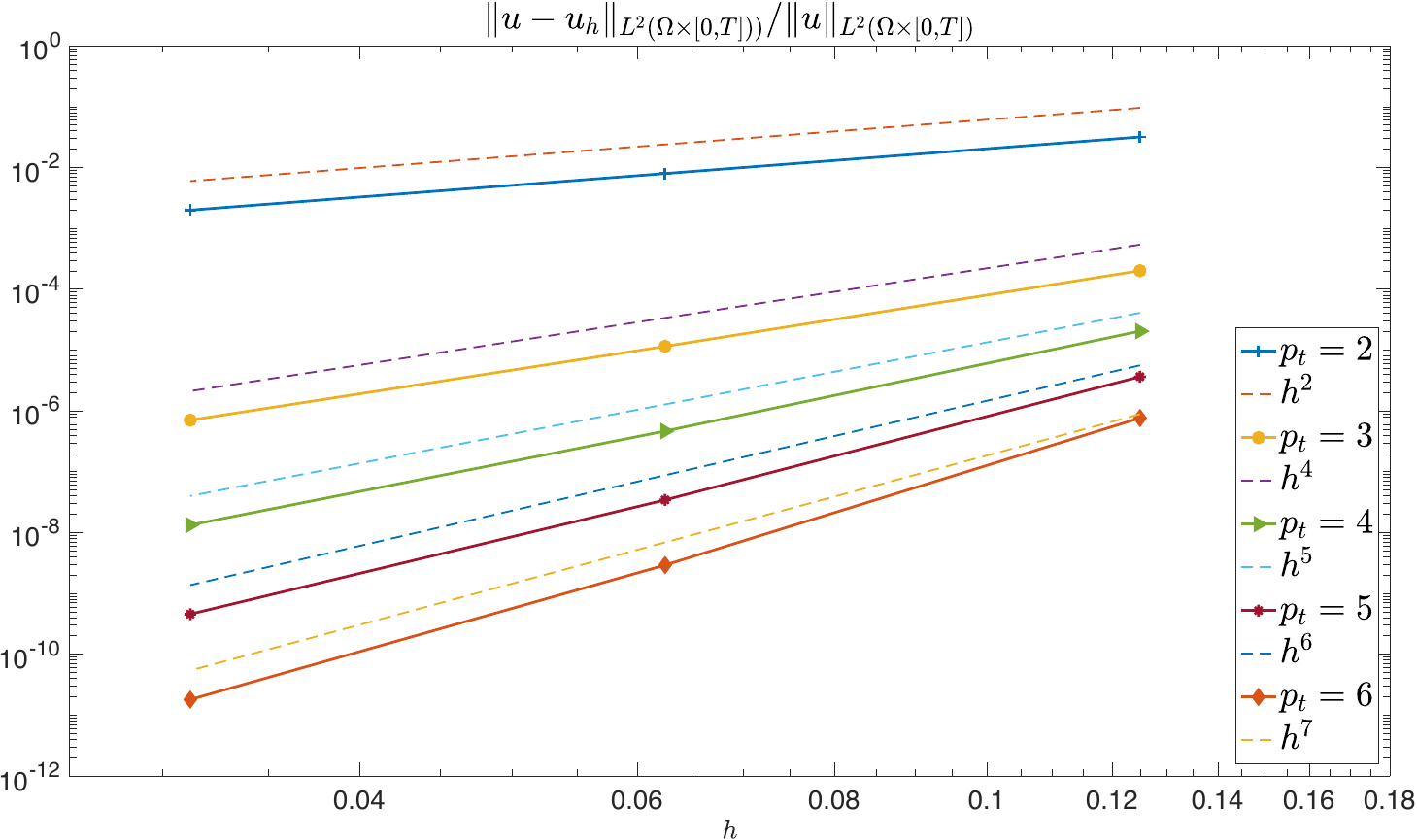}}\quad   
 \subfloat[][$H^1$-norm relative error with $p_t=p_s$.\label{fig:space-time_error_H1}]
   {\includegraphics[width=.48\textwidth]{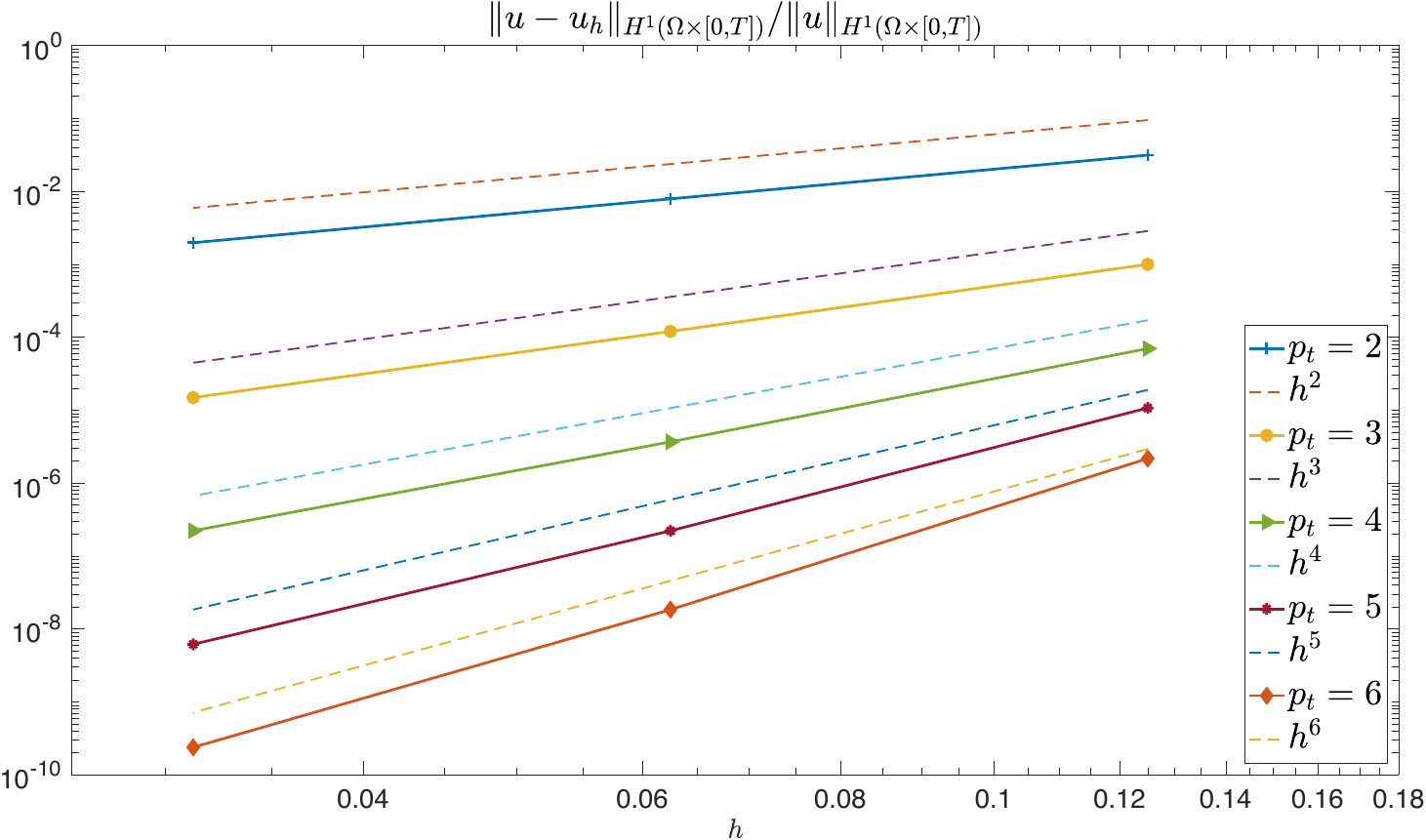}}
 \caption{Relative errors.}
 \label{fig:Errors}
\end{figure}

\subsection{Performance of the preconditioner} To assess the  performance of our preconditioning strategy, we  set $T=1$ and we focus on   two 3D spatial domains $\Omega \subset\mathbb{R}^3$, represented in Figure \ref{fig:cube} and Figure \ref{fig:rev_quarter}: the \emph{cube} and the \emph{rotated quarter of annulus}, respectively.  
As a comparison, we also consider  as preconditioner for CG the Incomplete Cholesky with zero fill-in  (IC(0)) factorization of $\vect{A}$, that is executed by the Matlab routine \texttt{ichol}. 
Tables \ref{tab:cube} and \ref{tab:rev_quarter} report the number of
iterations and the total solving time, that includes the setup time of
the preconditioner. The symbol `` * '' is used when  the  construction
of the  matrix $\vect{A}$ or its matrix factors  go out-of-memory. We force the execution to be sequential and to use only a single computational thread.

As discussed in the previous section, the matrix-vector products of CG are computed in a matrix-free way using its factors as in \eqref{eq:A_kron}. 
Matrix $\vect{A}$ is still assembled in order to use the IC(0) preconditioner. In any case, the assembly times are never included in the reported times.

 We first  consider  the domain $\widehat{\Omega}=\Omega=(0,1)^3$
 (Figure \ref{fig:cube}). Note that in this case we have that $[\vect{P}]_{i,j}=\mathcal{\widehat{P}}(\widehat{B}_{i,\vect{p}}, \widehat{B}_{j,\vect{p}})=\mathcal{P}(B_{i,\vect{p}}, B_{j,\vect{p}})$. We set homogeneous Dirichlet and zero initial boundary conditions and we fix $f$ such that the exact solution is $u = \sin(\pi x)\sin(\pi y)\sin(\pi z)\sin(t)$. 

Table \ref{tab:cube} shows the performance of   $\vect{P}$ and IC(0) preconditioners in the case $p_t = p_s$.  
 The number of iterations obtained with $\vect{P}$ are stable w.r.t
 $p_t$ and $n_{sub}$.  
 
Even if the number of iterations of our strategy might be larger
than that of IC(0), the overall computational time is significantly lower, 
up to two orders of magnitude for the problems considered. This is due to the higher setup and application cost of the IC(0) preconditioner.

{\renewcommand\arraystretch{1.2} 
\begin{table}
\begin{center}
\begin{tabular}{|r|c|c|c|c|}
\hline
& \multicolumn{4}{|c|}{$\vect{P}$ + CG \ $p_t=p_s$ \ Iterations /  Time } \\
\hline
$n_{sub}$ & $p_t=2$ & $p_t=3$ & $p_t=4$ & $p_t=5$   \\
\hline
8 &   \z9 / \z\z\z0.06  &  11 / \z\z\z0.07   &  11  / \z\z0.18     &   11  / \z\z0.28    \\
\hline
16 &    11 /  \z\z\z0.27 & 11 /  \z\z\z0.69  & 12 / \z\z1.80  &    12 / \z\z3.80     \\
\hline
32 &  12 / \z\z\z5.10    & 12 / \z\z13.37  &  12  / \z27.31   & 12 / \z52.95   \\
\hline
64 & 13 / \z100.09 & 13 / \z227.93 & 13 / 458.86 & 13 / 924.44 \\
\hline
128 & 13 / 2012.94 & 13 / 4235.96 & $\ast$  &  $\ast$ \\
\hline
\end{tabular} 
\vskip 2mm
\begin{tabular}{|r|c|c|c|c|}
\hline
& \multicolumn{4}{|c|}{IC(0) + CG \ $p_t=p_s$ \ Iterations / Time } \\
\hline
$n_{sub}$ & $p_t=2$ & $p_t=3$ & $p_t=4$ & $p_t=5$  \\
\hline
8 &   \z9 / \z\z0.18  & \z7  / \z1.69   &  \z6 / \z14.04     &   \z6  /  \z\z80.39    \\
\hline
16 &   22 / \z\z5.01 & 16 / 45.54   & 12 / 355.99   &    10 /     1913.90 \\
\hline
32 &  64 / 157.05     & $\ast$  &   $\ast$  & $\ast$   \\
\hline
\end{tabular}
\caption{Cube domain with  $p_t=p_s$. Performance of $\vect{P}+CG$ (upper table) and of IC(0)+CG (lower table). }
\label{tab:cube}
\end{center}
\end{table}}

 Then we consider as computational  domain $\Omega$ a quarter of annulus with center in the origin, internal radius 1 and external radius 2, rotated   along the axis $y=-1$ by $\pi/2$ 
(see Figure
\ref{fig:rev_quarter}). Boundary data and forcing function are set  such that the exact solution is $u = -(x^2+y^2-1)(x^2+y^2-4)xy^2\sin(z)\sin(t)$.

Table \ref{tab:rev_quarter} shows the results of CG coupled with $\vect{P}$, $\vect{P}^{\mathbf{G}}$ or IC(0) preconditioner. From the spectral estimates of Theorem \ref{teo:spec_est}, we know that the geometry parametrization $\mathbf{G}$, which in this case is not trivial, plays a key-role in the performance of $\vect{P}$. This is confirmed by the results of   Table \ref{tab:rev_quarter}: the number of iterations  is higher than the ones obtained in the cube domain, where $\mathbf{G}$ is the identity map (see Table \ref{tab:cube}).
 However,  the inclusion of some geometry information, and thus the use of $\vect{P}^{\mathbf{G}}$ as a preconditioner, improves the performances, as we can see from the middle table of Table \ref{tab:rev_quarter}.
Moreover, we show  that IC(0) is not competitive neither with
$\vect{P}$ nor with $\vect{P}^{\mathbf{G}}$, in terms of 
computational time.

For the last domain, we   analyze   the percentage of computation
time of  a $\vect{P}^{\mathbf{G}}$ application with respect to the
overall CG   time.  The results, reported in Table
\ref{tab:appl_matrix}, show that the time spent in the preconditioner
application takes only a little amount of the overall solving
time. The dominant cost, in this implementation  is due to the
matrix-vector products of the residual computation, that is the other
main operation performed in a CG cycle.  

Since we are primarily interested in the preconditioner
performance, in Figure \ref{fig:setup-appl}  we report in a log-log scale the computational times
required for the setup and for a single application of $\vect{P}^{\mathbf{G}}$ versus the number of degrees-of-freedom.  
We see that the setup time is clearly asymptotically proportional to
$N_{dof}$, as expected. Remarkably, the single application time grows slower than the expected theoretical cost $O(N_{dof}^{5/4})$; indeed, it grows almost as the optimal rate $O(N_{dof})$, even for the largest problems tested. As already mentioned, this is likely due to the high efficiency of the BLAS level 3 routines that perform the computational core of the application of the preconditioner.

{\renewcommand\arraystretch{1.28} 
\begin{table}
\begin{center} 
\begin{tabular}{|r|c|c|c|c|}
\hline
& \multicolumn{4}{|c|}{$\vect{P}$ + CG \ $p_t=p_s$ \ Iterations /  Time } \\
\hline
$n_{sub}$ & $p_t=2$ & $p_t=3$ & $p_t=4$ & $p_t=5$   \\
\hline 
8 &  107  / \z\z\z\z0.21 & 107 /  \z\z\z\z0.48  &  114 /   \z\z\z1.17   &   123  / \z\z\z2.73   \\
\hline
16 &  126  / \z\z\z\z2.56 & 128 / \z\z\z\z6.90   & 133 / \z\z17.04  &  135  /  \z\z35.177   \\
\hline
32 &  142 / \z\z\z52.77  & 143 / \z\z132.24 &   148 / \z292.53  &  151 / \z572.84 \\
\hline
64 & 153 / \z1056.21  & 155 / \z2415.23  & 156 / 4956.68 & 159 /  9906.33  \\
\hline
128 & 164 / 22106.01  & 166 / 47539.02    & $\ast$ & $\ast$ \\
\hline
\end{tabular} 
\vskip 2mm
\begin{tabular}{|r|c|c|c|c|}
\hline
& \multicolumn{4}{|c|}{$\vect{P}^{\mathbf{G}}$ + CG \ $p_t=p_s$ \ Iterations / Time } \\
\hline
$n_{sub}$ & $p_t=2$ & $p_t=3$ & $p_t=4$ & $p_t=5$  \\
\hline
8 &   24 / \z\z\z0.09 & 24  / \z\z\z\z0.13  &  26  /  \z\z\z0.37     &  26 /  \z\z\z0.60     \\
\hline
16 &   35 / \z\z\z0.77  & 34 / \z\z\z\z1.96    &  33 / \z\z\z4.62   &      33 /  \z\z\z9.35    \\
\hline
32 &   42 / \z\z17.03    & 41 / \z\z\z39.57  &  40 / \z\z82.35  & 41 / \z161.73  \\
\hline
64 & 46 / \z333.20 & 44 / \z\z716.03 & 49 / 1577.55 & 53 / 3384.08  \\
\hline
128 & 48 / 6767.08   & 50 / 14814.09  & $\ast$ & $\ast$\\
\hline
\end{tabular}
\vskip 2mm
\begin{tabular}{|r|c|c|c|c|}
\hline
& \multicolumn{4}{|c|}{IC(0) + CG \ $p_t=p_s$ \ Iterations / Time } \\
\hline
$n_{sub}$ & $p_t=2$ & $p_t=3$ & $p_t=4$ & $p_t=5$  \\

\hline
8 &   11 / \z\z0.17 &  \z8 / \z1.71   &   \z7 /  \z13.96     &   \z6 /  \z\z80.28     \\
\hline
16 &   29 / \z\z5.52  & 18 / 45.22    &  14 / 377.47   &      11 /   1895.55   \\
\hline
32 &  86 / 185.08    & $\ast$  &   $\ast$  & $\ast$  \\
\hline
\end{tabular}
\caption{Rotated quarter domain with  $p_t=p_s$. Performance of $\vect{P}+CG$ (upper table), $\vect{P}^{\mathbf{G}}+CG$ (middle table)  and of IC(0)+CG (lower table). }
\label{tab:rev_quarter}
\end{center}
\end{table}}
 
\renewcommand\arraystretch{1.28} 
\begin{table}
\begin{center}
\begin{tabular}{|r|c|c|c|c|}
\hline
& \multicolumn{4}{|c|}{$\vect{P}^{\mathbf{G}}$} \\
\hline
$n_{sub}$ & $p_t=2$ & $p_t=3$ & $p_t=4$ & $p_t=5$ \\
\hline
8 &      35.86\%  &     20.66\%   &      10.85\%     &     \z7.05   \% \\
\hline
16 &     17.90\%  &   \z8.10  \%  &    \z3.95  \%  &   \z2.28  \%    \\
\hline
32 &     14.25 \%  &  \z7.35  \%  &   \z4.05  \%   & \z2.49   \%      \\
\hline
64 & 17.28 \% &  \z8.75 \%  & \z4.67  \%  & \z2.52 \%  \\
\hline
128 &   23.98 \% &  12.21 \% & $\ast$ & $\ast$ \\
\hline
\end{tabular}  

\caption{Rotated quarter domain with $p_t = p_s$. Percentage of computational time of the preconditioner $\vect{P}^{\mathbf{G}}$ application in the overall
CG cycle.}
\label{tab:appl_matrix}
\end{center}
\end{table}

\begin{figure}
 \centering
\includegraphics[trim={0.1cm 0.1 0.3cm 0.5},clip,width=.8\textwidth]{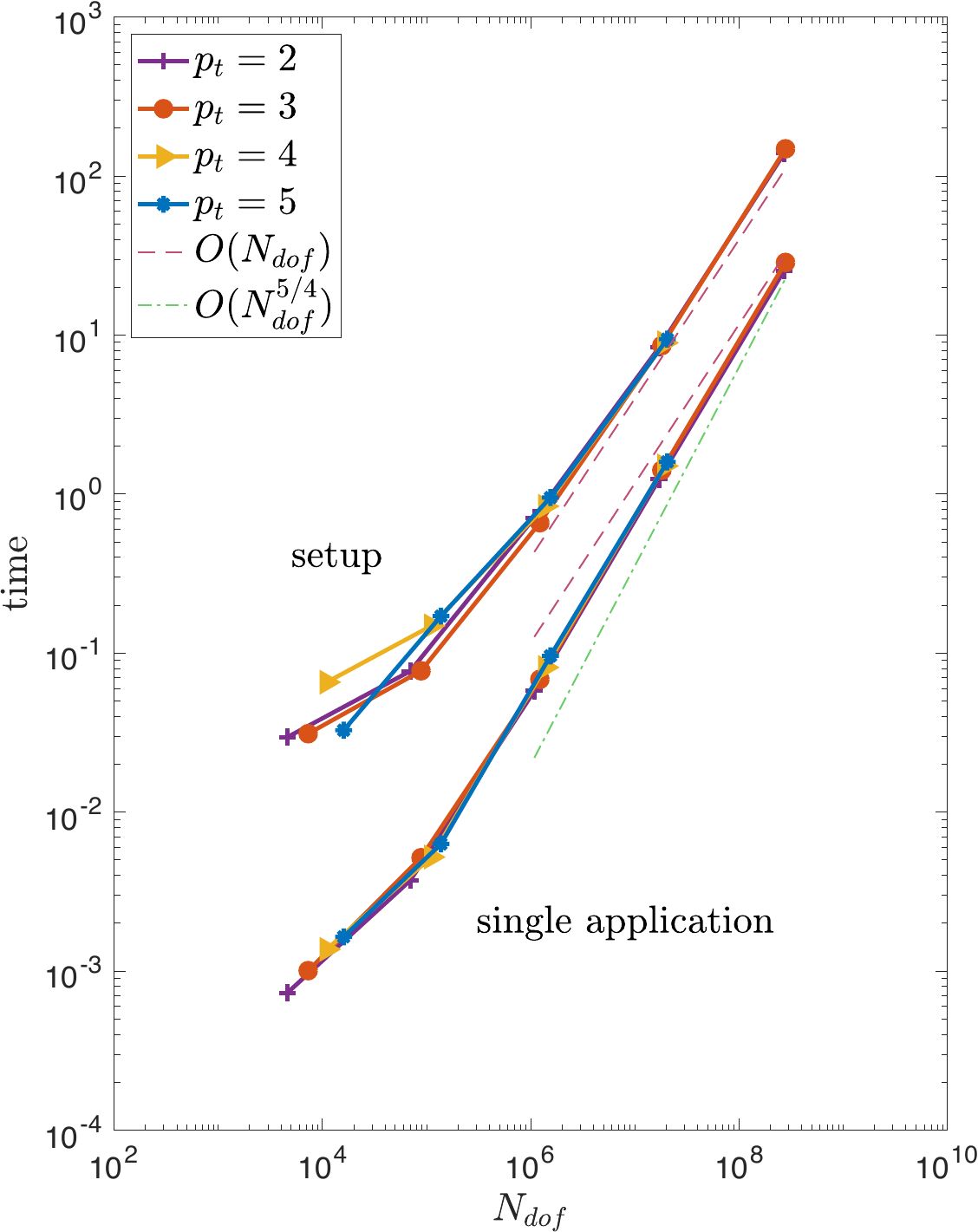}
 \caption{Rotated quarter domain with $p_t = p_s$. Setup times and  single application times of $\vect{P}^{\mathbf{G}}$.}
 \label{fig:setup-appl}
\end{figure}

\section{Conclusions} 

In this paper, we have proposed and studied a least-squares method 
for the heat equation, that  allows us to design an innovative  preconditioner  in the
framework of isogeometric analysis. 
 Even though  we adopt a global-in-time space-time  formulation,
based on smooth splines in space and time, the preconditioner
$\vect{P}$ that we have presented is highly efficient  both in terms of FLOPs and memory, thanks
to  its matrix representation as  suitable sum of Kronecker products,
leading to a Sylvester-like problem.

 The computational cost of the preconditioner
setup is at most $O(N_{dof})$ FLOPs while its  application  is  $O(N_{dof}^{1 +
  1/d})$ FLOPs. In
our numerical benchmarks the computational time, for  serial single-core
execution,  is in fact close to $O(N_{dof})$,
with no  dependence on  $p$.  The proposed  preconditioner $\vect{P}$
is indeed  robust with respect to the spline degree and
its variant, denoted with $\vect{P}^{\mathbf{G}}$, has a good performance
also when the geometry parametrization  ${\mathbf{G}}$  of the patch is not trivial. 

The storage cost is instead $O(p^d N_s + N_{dof})$, under the reasonable  assumption that
$n^2_t \leq C p^d N_s $. We emphasize that is roughly the same storage
cost that one would get by discretizing separately in space and time.

Our approach could  be coupled with a matrix-free
 idea  (see  \cite{Sangalli2017}),
and  this is expected to further  improve  the 
efficiency  of the overall method. Everything is well-suited for
parallelization: even though  in this paper
we do not consider parallel implementation, 
this  is a promising research direction for the future.

\appendix
\section{Appendix}
 
\subsection{Smooth approximation of $\widehat{\V}_0$}\label{App-smooth-approximation}

In this appendix we prove the density of spaces of  smooth functions, with
boundary conditions, in suitable Sobolev spaces on the parametric domain. The first result
concerns $H^1_0 \cap H^2$.  

\begin{lemma}\label{lemma:smooth-approx-1}
  Let $Q:=\widehat{\Omega}\times (a,b)$ be an open $(d+1)$-dimensional box.     Then, the space $ C^\infty \left ( {\overline{Q}} \right ) \cap H^1_0 \left ( { {Q}}  \right) $ is dense in  $ H^2 \left (
    {  {Q}} \right ) \cap H^1_0 \left (  {  {Q}} \right ) $.
\end{lemma}

\begin{proof} Let $w \in   H^2 \left (
    {   {Q}} \right ) \cap H^1_0 \left(  {  {Q}} \right ) $ and $g := -\Delta w \in  L^2 \left (  {   {Q}} \right ) $. Clearly, $w$ solves, in a weak sense,   
\[
\begin{cases}
    - \Delta w   = g & \text{in $  {   {Q}}  $,} \\
  \hskip 5mm  w   = 0  & \text{on $\partial  {   {Q}}  $.}
    
\end{cases}
\]
Let $g_n \in C^\infty_0 \left (  {  {Q}} \right ) $ such that $g_n
\to g$ in $L^2 ({  {Q}}  )$ and let $w_n \in H^1_0\left
  ( {  Q} \right)$ be the  weak   solution of 
 
\[
\begin{cases}
    - \Delta w_n = g_n & \text{in $  {  {Q}}  $,} \\
   \hskip 5.2mm  w_n = 0  &\text{on $\partial  {   {Q}}  $.}
    
\end{cases}
\]

 Then
$w_n \to w$ in $H^2\left ( {  {Q}}\right )$. Note that $w_n$ is
defined on $ { \overline{Q}}$, 
vanishes on  its boundary  $\partial { {Q}}$ and  is harmonic in a inner neighborhood
of $\partial {  {Q}}$ because $g_n$ has compact support, thus,   employing recursively Schwarz
reflection (see, e.g., \cite[Exercise 9, Section 2.5]{Evans2010book}
and \cite[Remarque 10, Section IX.2]{brezis1983analyse}) we  can
extend  $w_n$ outside $ { \overline{Q}}$, such that this extension is
harmonic in a neighborhood of $\partial { {Q}}$. It follows that $w_n \in C^\infty \left ( {  \overline{Q}}\right )  $. \end{proof} 

The second result focuses on the space which is needed for our
least-squares formulation, that is, $H^2$ in space and $H^1$ in time,
endowed with homogeneous initial and boundary conditions. This is used to
show, in Theorem~\ref{teo:convergence},  the convergence of our method.

\begin{lemma}\label{lemma:smooth-approx-2}
 Let  
 \begin{align*}
\widehat{\V}_{ 0} =& \left\{ v \in \left[ \left  ( H^2  ( { \widehat{\Omega}} )\cap
  H^1_0  ( { \widehat{\Omega}}  )\right )\otimes L^2(0,1)\right] \cap
\left[  L^2 ( { \widehat{\Omega}}  )\otimes H^1(0,1) \right]\
 \text{ s.t. }  \ \ v=0   \text{ on }  { \widehat{\Omega}}\times\{0\}  \right\}  \end{align*}
  be a Hilbert space endowed with the norm
 
\[
\|v\|_{\widehat{\V}_0}^2:= 
\int_{0}^1\|\Delta v(\cdot, \tau)  \|_{L^2(\widehat\Omega)}^2  \,\d\tau +  \int_{0}^1 \|\partial_\tau v(\cdot, \tau) \|_{L^2(\widehat{\Omega})}^2\,\d\tau .
\]
 
Then, the space $ C^\infty \left ( [0,1]^{d+1}\right ) \cap
\widehat{\V}_{ 0}  $  is dense in $ \widehat{\V}_{ 0}  $. 
\end{lemma}

 \begin{proof}

Consider a given $w \in \widehat{\V}_{ 0}$ as the solution of  a heat
problem   on the parametric domain $\widehat{\Omega} \times (0,1) = (0,1)^{d+1}$,  with datum $g := ( \partial_t w - \Delta w)\in L^2\left ( (0,1)^{d+1} \right ) $,     i.e. 

\begin{equation}
\label{eq:heat_eq_param}
\begin{cases} 
		\partial_t w - \Delta w \ = \  g  \quad &\mbox{in }\ \ \ \hspace{1mm} \widehat{\Omega} \times (0, 1),  \\
		 w  \ =\  0  \quad \quad & \mbox{on }\  \hspace{0.7mm}   \partial \widehat{\Omega}   \times  (0,1), \\
	 	 w \ =\  0\quad & \mbox{in }\ \  \ \ \widehat{\Omega}   \times \lbrace 0  \rbrace. 
\end{cases}
\end{equation}

  Let $g_n \in C^\infty_0 \left (
  (0,1)^{d+1} \right ) $ such that $g_n \to g $ in $L^2\left (
  (0,1)^{d+1} \right )  $ and let $w_n$ be the solution of the same heat
problem   \eqref{eq:heat_eq_param}   with datum $g_n$. Following the proof of    Theorem
\ref{teo:parabolic-regularity} applied in ${ \widehat{\Omega}}\times (0,1),$   we get $w_n \to w$ in $\widehat{\V}_{ 0}$
while, by \cite[Section 7.1.3, Theorem 6]{Evans2010book}, we also get $w_n \in  L^2( { \widehat{\Omega}})\otimes H^2 ( 0,1)$. 

We use now Lemma~\ref{lemma:smooth-approx-1} to approximate $w_n$.   Fix $\delta > 0$   and consider an extension $z_n$ of $w_n$ and $z_n \in  \left[ \left  ( H^2 ( {  \widehat{\Omega}}  )\cap
  H^1_0( { \widehat{\Omega}}  )\right )\otimes L^2(0,1+\delta)\right] \cap
\left[  L^2  ( { \widehat{\Omega}}  )\otimes H^2(0,1+\delta) \right]$ such
  that $z_n (\cdot, 1+\delta ) = 0$.\footnote{ The extension is obtained, for instance, in the following
  way. Consider the null extensions $\widetilde{f}_n$ of $f_n$ in $L^2
  \left ( {  \widehat{\Omega}} \times(0,1+\delta) \right ) $. Let
  $\widetilde{w}_n$ be the solutions of a heat problem  \eqref{eq:heat_eq_param}   in ${ \widehat{\Omega}}
  \times(0,1+\delta) $ (note that $\widetilde{w}_n$ is an extension
  of $w_n$, by uniqueness,  and that $\widetilde{w}_n$ has the same
  regularity of $w$).  Next, let $\phi$ be a cut-off function for
  $(0,1)$ in $(0, 1+\delta)$ and let $z_n ( \cdot  t) = \phi (t)  \, \widetilde{w}_n (\cdot, t)$.}  
Now   observe that $z_n$ is a function in $ \in H^2 
    \left (  \widehat{\Omega} \times(0,1+\delta) \right )  \cap H^1_0  \left
      (  \widehat{\Omega}  \times(0,1+\delta) \right )  $,   we can then apply  Lemma \ref{lemma:smooth-approx-1}  to construct a
    sequence ${z}_{n,k} \in   C^\infty \left({  [0,1]^d} \times [0,1+\delta] \right )\cap H^1_0\left(\widehat{\Omega}\times(1+\delta)\right) 
   $  converging, as $k \rightarrow \infty$, to  $z_n $  in the $ H^2 
    \left (  \widehat{\Omega} \times(0,1+\delta) \right ) $ norm. The
    restriction of   ${z}_{n,k} $ to $[0,1]^{d+1} $ belongs to the
    required space $   C^\infty \left (  {  [0,1]^{d+1}} \right )\cap \widehat{\V}_0   $ and the sequence
converges (as $k \rightarrow \infty$) to  $z_n $ in the
$H^2((0,1)^{d+1})$ norm, and thus in the  $\widehat{\V}_0$-norm.   
\end{proof}

\subsection{A variational formulation  equivalent to \eqref{eq:min_prob}--\eqref{eq:cont_problem} }\label{App}

\renewcommand{\dx}{\mathrm{d} \Omega}
 
In this appendix, we show that the least-squares space-time functional
\begin{equation}
\label{eq:ls-energy}
\mathcal{E}_{LS}(v):=\tfrac12 \int_0^T    \| \partial_t v(\cdot,t)   - \Delta v(\cdot,t) - f(\cdot,t)   \|^2_{L^2(\Omega)}  \,\dt \quad \forall v\in \mathcal{V}_0
\end{equation}
that appears in the minimization problem   \eqref{eq:min_prob}, coincides with another space-time functional  \eqref{eq:energy}   appearing in the theory of gradient flows and curves of maximal slopes  (see e.g., \cite{Ambrosio2008, Santambrogio2017}).

First, let us introduce the energy $\mathcal{J} : H^1_0 ( \Omega)\times [0,T] \to \mathbb{R}$ given by  $$ \mathcal{J} (  w, t ) := \int_\Omega\left(   \tfrac12 |  \nabla w ( \vecx) |^2 - f(\vecx , t)  w (\vecx ) \right) \, \dx  $$
 and assume, for the sake of simplicity, that $f \in H^1 ( 0, T ; L^2 (\Omega)) = L^2(\Omega) \otimes H^1(0,T)$. If $w \in    H^1_0(\Omega)\cap H_\Delta (\Omega)$   then for all $z \in H^1_0 ( \Omega)$ and for all $t\in (0,T)$ by Green's formula we have 
 \begin{equation*} 
\partial_w \mathcal{J} ( w, t ) [ z ] = \int_\Omega \big( - \Delta w (\vecx) - f (\vecx , t)  \big) z (\vecx) \, \dx.
 \end{equation*}
Moreover, thanks to the regularity of $f$ we have 
 \begin{equation*} \partial_t \mathcal{J}  ( w, t ) =   - \int_\Omega w (\vecx)  \partial_t f (\vecx , t) \, \dx .
 \end{equation*}
At this point, let us see that the functional $\mathcal{E}_{LS}$ coincides with the following functional defined $\forall v\in \V_0$ 
\begin{align}
\label{eq:energy}
\mathcal{E}(v):=& \mathcal{J} (  v ( \cdot, T) , T ) + \tfrac12 \int_0^T \left(  \| \partial_t v (\cdot, t)  \|^2_{L^2(\Omega)} + \| \Delta v(\cdot, t) + f(\cdot, t)  \|^2_{L^2(\Omega)}  \right) \, \dt \,   - \int_0^T   \partial_t \mathcal{J} (v(\cdot, t), t )  \, \dt.\nonumber
\end{align}  

For $v \in \V_0$ we know, e.g., by \cite[Lemme 3.3]{Brezis1973}, that the energy $t \mapsto \mathcal{J} ( v(\cdot, t), t)$  
is absolutely continuous and thus
 \begin{align*}
  \mathcal{J} ( v (\cdot, T), T) & = \int_0^T \frac{\mathrm{d}}{\mathrm{d} t} \mathcal{J} ( v(\cdot, t), t) \, \dt  = 
	 \int_0^T \left( \partial_w \mathcal{J} ( v(\cdot, t), t) [ \partial_t v(\cdot, t) ] + \partial_t \mathcal{J} (v(\cdot, t),t ) \right)  \dt \\
	& = \int_0^T \int_\Omega \big( - \Delta v - f   \big) \partial_t v  \, \dx  \,\dt - 
		 \int_0^T  \int_\Omega v \,\partial_t f     \, \dx  \, \dt \,.
\end{align*} 
Then, we can re-write the least-squares functional \eqref{eq:ls-energy} as follows:
\begin{align} 
 \mathcal{E}_{LS}(v)& = \tfrac12 \int_0^T \left(  \| \partial_t v(\cdot, t) \|^2_{L^2(\Omega)} + \| \Delta v(\cdot, t)  + f(\cdot, t)  \|^2_{L^2(\Omega)} \right)  \, \dt   -  \int_0^T  \int_\Omega   ( \Delta v   +  f ) \partial_t v  \, \dx \, \dt  \\&  =  \mathcal{J} (  v (\cdot, T), T ) + \tfrac12 \int_0^T \left(  \| \partial_t v(\cdot, t) \|^2_{L^2(\Omega)} +  \| \Delta v(\cdot, t) + f(\cdot, t)  \|^2_{L^2(\Omega)}\right) \, \dt \,   - \int_0^T \hspace{-3pt}   \partial_t \mathcal{J} (v(\cdot, t) , t )  \, \dt \nonumber \\
 & =     
\mathcal{E} (v)\nonumber .
\end{align} 
 As a consequence, the representation \eqref{eq:equivalent_A} in the discrete space $\mathcal{V}_{h,0}$ holds also in the space $\mathcal{V}_0$, moreover, the bilinear form \eqref{eq:a-form} turns out to be the Euler-Lagrange equation of the functional \eqref{eq:energy}.

\subsection{Separation of variables algorithm}\label{app:sep_var}
 After approximating each function  by 
 piecewise constants, \eqref{eq:coefficient-approx}
 becomes
 \begin{equation}
\label{eq:coefficient-approx-tensor}
  \begin{aligned}
& \left[\mathfrak{C}^{(k)}\right]_{i_1,\dots,i_{d+1}}\approx
[{\mub}^{(1)}]_{i_1}\dots
[\mub^{(k-1)}]_{i_{k-1}}[\omegab^{(k)}]_{i_k}[\mub^{(k+1)}]_{i_{k+1}} \dots
 [\mub^{(d+1)}]_{i_{d+1}}, 
\end{aligned}
\end{equation}
where, denoting by  $\mathbb{R}_{+}$  the set of strictly positive real numbers, the tensors $\mathfrak{C}^{(k)}\in\mathbb{R}_{+}^{n_1\times\dots\times
  n_{d+1}}$ are given and  $\mub^{(k)}, \omegab^{(k)} \in
\mathbb{R}_{+}^{n_k}$, $k=1,\ldots,d+1$,  are  unknown vectors to be computed. In our case,  $n_1, \ldots, n_{d}$ are the number of
elements in each space direction and $n_{d+1}$ the number of elements
in time, and we construct $\mathfrak{C}^{(k)}\in\mathbb{R}_{+}^{n_1\times\dots\times
  n_{d+1}}$  by interpolating $c_k$ in the element barycenters.

In order to compute the approximation
\eqref{eq:coefficient-approx-tensor}, we aim at finding  $\mub^{(k)}, \omegab^{(k)} \in\mathbb{R}_{+}^{n_k}$  for $k=1,\dots,d+1$, that minimize the functional
\[ \begin{aligned}
  \left [ \vect{\chi}^{(k)}, \vect{\psi}^{(k)}\ \right ]_{k=1,\dots,d+1} 
\longmapsto  \   \max_{ \substack{ i_k=1,\dots, n_k;\\  k=1,\dots, d+1 } }\left\{ \left|\log  \left(\frac{[\mathfrak{C}^{(k)}]_{i_1,\dots,
          i_{d+1}}} {[\vect{\chi}^{(1)}]_{i_1}\dots
       [\vect{\chi}^{(k-1)}]_{i_{k-1}}[\vect{\psi}^{(k)}]_{i_k}[\vect{\chi}^{(k+1)}]_{i_{k+1}}
        \dots [\vect{\chi}^{(d+1)}]_{i_{d+1}}}\right)\right|  \right\}.
\end{aligned} \]
Equivalently, we look for $\mub^{(k)}, \omegab^{(k)} \in\mathbb{R}_{+}^{n_k}$ for $k=1,\dots,d+1$, such that the minimum and maximum values of the ratio 
\[ \frac{[\mathfrak{C}^{(k)}]_{i_1,\dots, i_{d+1}}}{[\mub^{(1)}]_{i_1}\dots
[\mub^{(k-1)}]_{i_{k-1}}[\omegab^{(k)}]_{i_k}[\mub^{(k+1)}]_{i_{k+1}} \dots [\mub^{(d+1)}]_{i_{d+1}}}, 
 \]
for $i_k=1,\dots, n_k; \;  k=1,\dots, d+1$, are as close as possible to 1 (in the logarithmic sense).

Algorithm \ref{al:sep_approx} computes an approximate solution of the
 above optimization problem. 
This algorithm generalizes the one
used in \cite{Wachspress1984} which is focused on the case of two variables, i.e. it computes the approximations
\[\left[\mathfrak{C}^{(1)}\right]_{i_1,i_2} \approx [\omegab^{(1)}]_{i_1}[{\mub}^{(2)}]_{i_2},
\qquad \left[\mathfrak{C}^{(2)}\right]_{i_1,i_2} \approx [{\mub}^{(1)}]_{i_1}[\omegab^{(2)}]_{i_2}.\]
Note that in this case the two approximation problems are completely decoupled, so they can be solved independently.
As in \cite{Wachspress1984}, in all our tests we set $maxit = 2$.
 
\begin{algorithm} 

\caption{Separation of variables}\label{al:sep_approx}
 \begin{algorithmic}[1]
  \State  Initialize $\mub^{(l)}=\omegab^{(l)}=\mathbf{1}_{n_l}$ for $l=1,\dots,d+1$.
  \For{ $iter=1 \dots maxit$}{
  \For{$k=1,\dots,{d+1}$}{
 \State Compute $\mathfrak{V}^{(k)}\in\mathbb{R}^{n_1\times\dots\times n_{d+1}}$ s.t.   
\Statex    \qquad \qquad  \quad \quad \quad \quad \qquad  $\left[\mathfrak{V}^{(k)}\right]_{i_1,\dots,i_{d+1}}=\frac{[\mathfrak{C}^{(k)}]_{i_1,\dots, i_{d+1}}} {[\mub^{(1)}]_{i_1}\dots [\mub^{(k-1)}]_{i_{k-1}}[\mub^{(k+1)}]_{i_{k+1}} \dots [\mub^{({d+1})}]_{i_{d+1}}}.$ 
      \For{$j=1,\dots,n_k$}  
\State Compute $m=\min\left\{ \mathfrak{V}^{(k)} _{i_1,\dots,i_{k-1},j,i_{k+1},\dots i_{{d+1}}} \text{ s.t. }  i_l  =1,\dots,n_l; l=1,\dots,{d+1} \text{ and } l\neq k \right\}.$ 
\State Compute $M=\max\left\{ \mathfrak{V}^{(k)} _{i_1,\dots,i_{k-1},j,i_{k+1},\dots i_{{d+1}}}   \text{ s.t. }    i_l  =1,\dots,n_l; l=1,\dots,{d+1} \text{ and } l\neq k \right\}.$ 
\State Update $ [\omegab^{(k)}]_j=\sqrt{mM}. $
            \EndFor   
        }\EndFor
        \For{$k=1,\dots,{d+1}$}{
             \For{$l=1,\dots,{d+1}$} 
            \If{$l\neq k$} \State Compute $\mathfrak{W}^{(k,l)}\in\mathbb{R}^{n_1\times\dots\times n_{d+1}}$  s.t. 
            \Statex  \qquad \qquad  \quad \quad \quad \quad \qquad $\left[\mathfrak{W}^{(k,l)}\right]_{i_1,\dots,i_{d+1}}=\frac{[\mathfrak{C}^{(k)}]_{i_1,\dots, i_{d+1}}[\mub^{(l)}]_{i_l}} {[\mub^{(1)}]_{i_1}\dots [\mub^{(k-1)}]_{i_{k-1}}[\omegab^{(k)}]_{i_k}[\mub^{(k+1)}]_{i_{k+1}} \dots [\mub^{({d+1})}]_{i_{d+1}}}.$ \EndIf\EndFor 
             \State Compute
             $\mathfrak{Y}\in\mathbb{R}^{n_1\times\dots\times
               n_{d+1}}$
             s.t.  $[\mathfrak{Y}]_{i_1,\dots,i_{n_{d+1}}}=\min\left\{[\mathfrak{W}^{(k,l)}]_{i_1,\dots,i_{n_{d+1}}}
                \text{ s.t. }  l=1,\dots, {d+1} \text{ and } l\neq k\right\}$
                         \State Compute
            $\mathfrak{Z}\in\mathbb{R}^{n_1\times\dots\times n_{d+1}}$
            s.t.      $[\mathfrak{Z}]_{i_1,\dots,i_{n_{d+1}}}=\max\left\{[\mathfrak{W}^{(k,l)}]_{i_1,\dots,i_{n_{d+1}}}
               \text{ s.t. } l=1,\dots, {d+1} \text{ and } l\neq k\right\}$
      \For{$j=1,\dots,n_k$} 
      \State Compute $m=\min\left\{ [\mathfrak{Y}]_{i_1,\dots,i_{k-1},j,i_{k+1},\dots i_{{d+1}}}  \text{ s.t. }     i_l  =1,\dots,n_l;l=1,\dots,{d+1} \text{ and } l\neq k \right\}.$ 
\State Compute $M=\max\left\{ [\mathfrak{Z}]_{i_1,\dots,i_{k-1},j,i_{k+1},\dots i_{{d+1}}}  \text{ s.t. }     i_l  =1,\dots,n_l;l=1,\dots,{d+1} \text{ and } l\neq k \right\}.$
\State Update $ [\mub^{(k)}]_j=\sqrt{mM}. $  
   \EndFor  
            }\EndFor
  }\EndFor 
\end{algorithmic}
\end{algorithm}


\bibliographystyle{plain}
 \bibliography{biblio_space_time}

\end{document}